\documentclass[11pt]{article}
\usepackage{}

\usepackage[total={6in, 8.5in}]{geometry}
\usepackage{amsfonts}
\usepackage{amsthm}
\usepackage{amssymb}
\usepackage{amsmath}
\usepackage{enumerate}
\usepackage[
pdfauthor={},
pdftitle={The Core Conjecture of Hilton and Zhao II: a Proof},
pdfstartview=XYZ,
bookmarks=true,
colorlinks=true,
linkcolor=blue,
urlcolor=blue,
citecolor=blue,
bookmarks=false,
linktocpage=true,
hyperindex=true
]{hyperref}

\makeatletter
\newcommand{\setword}[2]{%
	\phantomsection
	#1\def\@currentlabel{\unexpanded{#1}}\label{#2}%
}
\makeatother

\makeatletter
\renewcommand*\env@matrix[1][*\c@MaxMatrixCols c]{%
	\hskip -\arraycolsep
	\let\@ifnextchar\new@ifnextchar
	\array{#1}}
\makeatother

\usepackage{graphicx}
\usepackage[natural]{xcolor}

\usepackage{pgf,tikz,pgfplots}

\usepackage{mathrsfs}
\usepackage{cases}

\usetikzlibrary{arrows}

\usepackage{xcolor}
\long\def\ignore#1{}

\let\oldi\ignore

 \oldi{
\newtheorem{THM}{\textbf{Theorem}}[section]
\newtheorem{THMs}{\textbf{Theorem}}[section]
\newtheorem{DEF}[THM]{\textbf{Definition}}[section]
\newtheorem{LEM}[THM]{\textbf{Lemma}}
\newtheorem{CON}[THM]{\textbf{Conjecture}}
\newtheorem{PROP}[THM]{\textbf{Proposition}}
\newtheorem{COR}[THM]{\textbf{Corollary}}
\newtheorem{CORs}{\textbf{Corollary}}[section]
\newtheorem{PRO}[THM]{\textbf{Problem}}
\newcommand{\pf}{\textbf{Proof}.\quad}

\newtheorem{FAC}{\textbf{Fact}}
\newtheorem{REM}{\textbf{Remark}}
\newtheorem{OPR}{\textbf{Operation}}
\newtheorem{CLA}{\textbf{Claim}}[section]
\setcounter{CLA}{1}
 }

\newtheorem{THM}{Theorem}[section]

\newtheorem{LEM}[THM]{Lemma}
\newtheorem{CON}[THM]{Conjecture}

\newtheorem{COR}[THM]{Corollary}

\newtheorem{CLA}{Claim}[section]
\newcommand{\pf}{\textbf{Proof}.\quad}

\newtheorem*{THM2}{\textbf{Theorem 2.4}}
\newtheorem*{THM3}{\textbf{Theorem 2.5}}
\newtheorem*{Ass}{\textbf{Assumption}}

\usepackage{thmtools}
\usepackage{thm-restate}

\newcommand{\CC}{\mathcal{C}}

\newcommand{\Of}{\mathcal{O}}

\newcommand{\pbar}{\overline{\varphi}}
\newcommand{\arxiv}[1]{\href{http://arxiv.org/abs/#1}{\texttt{arXiv:#1}}}
\begin{document}
\title{The Core Conjecture of Hilton and Zhao II: a Proof}

\author{%
 Yan Cao\thanks{Department of Mathematics, 
 	West Virginia University, Morgantown, WV 26506, USA.  \texttt{yacao@mail.wvu.edu}.}
 \quad Guantao Chen\thanks{Department of Mathematics and Statistics, 
 	Georgia State University, Atlanta, GA 30302, USA.  \texttt{gchen@gsu.edu}.  Partially supported by NSF grant DMS-1855716.}\\
 \quad 
 Guangming Jing\thanks{ Department of Mathematics,
 	Augusta University, Augusta, GA 30912, USA. 
 	\texttt{gjing@augusta.edu}.  Partially supported by NSF grant DMS-2001130.}
 \quad 
 Songling Shan\thanks{Department of Mathematics, 
 	Illinois State Univeristy, Normal, IL 61790, USA. 
 	\texttt{sshan12@ilstu.edu}.  Partially supported by the NSF-AWM Mentoring Travel Grant 1642548 and by the New Faculty Initiative Grant of Illinois State University.}
 } 

\date{\today}
\maketitle

 \begin{abstract}
 A simple graph  $G$ with maximum degree $\Delta$ is  \emph{overfull} if $|E(G)|>\Delta \lfloor |V(G)|/2\rfloor$. The \emph{core} of $G$, denoted $G_{\Delta}$, is the subgraph of $G$ induced by its vertices of degree $\Delta$.  Clearly, the chromatic index of $G$ equals  $\Delta+1$ if $G$ is overfull. 
Conversely, Hilton and Zhao in 1996 conjectured  that if $G$ is a simple connected graph with $\Delta\ge 3$ and  $\Delta(G_\Delta)\le 2$, then $\chi'(G)=\Delta+1$  implies that $G$ is overfull or $G=P^*$, where $P^*$ is obtained from the Petersen graph by deleting a vertex.  Cariolaro and Cariolaro  settled the
base case $\Delta=3$ in 2003, and Cranston and Rabern proved  the next case $\Delta=4$ in 2019.   In this paper, we give a proof of this conjecture for all $\Delta\ge 4$.


 \end{abstract}


\section{Introduction}

Let $G$ be a simple graph with maximum degree $\Delta$. 
The \emph{core} of  $G$,  denoted $G_\Delta$, 
is the subgraph of $G$ induced by its vertices of degree $\Delta$.  If $|E(G)|>\Delta \lfloor |V(G)|/2\rfloor$, then $G$
is \emph{overfull}. Overfull graphs are class 2. The graph $P^*$,  
obtained from the Petersen graph by deleting one vertex,  
 is  also known to be class 2. 
Conversely,  in 1996, Hilton and Zhao~\cite{MR1395947} proposed the following  conjecture.
\begin{CON}[Core Conjecture]\label{Core Conjecture}
	Let $G$ be a simple       
	connected  graph with maximum degree $\Delta\ge 3$ and $\Delta(G_\Delta)\le 2$. 
	Then $G$ is class 2 implies that $G$ is overfull or $G=P^*$. 
\end{CON}

 As a class 2 graph of maximum degree 2 is an odd cycle
 and odd cycles are overfull,  if true, the Core Conjecture implies that 
for connected graphs $G$ with $\Delta(G_\Delta)\le 2$, 
determining whether $G$ is class 2 can be done by checking whether 
  $|E(G)| > \Delta \lfloor |V(G)|/2\rfloor$  if $G\ne P^*$. 
We call a connected class 2 graph $G$ with $\Delta(G_\Delta)\leq 2$ an \emph{HZ-graph}.   
A first breakthrough of the Core Conjecture was achieved in 2003, when Cariolaro and Cariolaro \cite{CariolaroC2003} settled the
base case $\Delta=3$. They proved that $P^*$ is the only HZ-graph with maximum degree $\Delta=3$, an alternative proof was given later by Kr\'al', Sereny, and Stiebitz (see \cite[pp. 67--63]{StiebSTF-Book}). The next case $\Delta=4$ was recently solved by Cranston and Rabern \cite{CranstonR2018hilton}: they proved that the only HZ-graph with maximum degree $\Delta=4$ is the graph  $K_5$ with one edge removed. In this paper, we confirm 
the Core Conjecture for all HZ-graphs $G$ with $\Delta\ge 4$. 
It worth mentioning that our proof implies a polynomial-time algorithm that, given $G$ with  maximum degree $\Delta \ge 4$ and $\Delta(G_\Delta)\le 2$, finds an optimal edge coloring of $G$. 
\begin{THM}\label{Thm:main}
Let  $G$ be a connected  graph with maximum degree $\Delta \ge 4$  and $\Delta(G_\Delta)\le 2$.  Then  $G$ is class 2 if and only if $G$ is overfull. 
\end{THM}

Since every overfull graph is class 2, we will only prove the ``only if'' statement in Theorem~\ref{Thm:main}. The remainder of the paper is organized as follows. In next section, we prove 
Theorem~\ref{Thm:main} by applying Theorems~\ref{Thm:vizing-fan} to~\ref{Thm:nonadj_Delta_vertex}. 
In Section 3, we  give necessary definitions and list results from~\cite{HZI}. 
Theorems~\ref{Thm:vizing-fan} to~\ref{Thm:nonadj_Delta_vertex}
will be proved in Sections 4, 5, and 6, respectively. 
 
\section{Proof of Theorem~\ref{Thm:main}} 

In this section, we prove Theorem~\ref{Thm:main} by applying  Theorems~\ref{Thm:vizing-fan} to~\ref{Thm:nonadj_Delta_vertex}. 
We start with some  concepts. 
For two integers $p$ and $q$, let $[p,q]=\{i\in \mathbb{Z}: p\le i\le q\}$.
An  \emph{edge $k$-coloring} of $G$ is a mapping $\varphi$ from $E(G)$ to 
$[1,k]$, called \emph{colors}, such that  no two adjacent edges receive the same color.  We denote by  $\CC^k(G)$  the set of all edge $k$-colorings of $G$. 
The  \emph{chromatic index}  $\chi'(G)$ of $G$ is the smallest  $k$ so that $G$ has an edge $k$-coloring.  The symbol $\Delta$  is reserved for $\Delta(G)$, the maximum degree of $G$
throughout  this paper.

Let $G$ be a graph, $v\in V(G)$, and $i\ge 0$ be an integer.  
An \emph{$i$-vertex}  is a vertex of degree $i$
in $G$, and an $i$-vertex from the neighborhood of   $v$ is called an \emph{$i$-neighbor}  of $v$. 
Define 
$$ V_i=\{w\in V(G)\,:\, d_G(w)=i\} , \quad \quad N_{i}(v)=N_G(v)\cap V_i, \quad \,\mbox{and} \quad N_i[v]=N_i(v)\cup \{v\}. $$
For  $X\subseteq V(G)$,  let $\mathit {N_{G}(X)=\bigcup_{x\in X}N_G(x)}$ 
and $ \mathit {N_i(X)=N_G(X)\cap V_i}$.

Let  $e\in E(G)$ and 
$\varphi\in \CC^k(G-e)$ for  some integer $k\ge 0$. 
The set of colors \emph{present} at $v$ is 
$\varphi(v)=\{\varphi(f)\,:\, \text{$f$ is incident to $v$}\}$, and the set of colors \emph{missing} at $v$ is $\pbar(v)=[1,k]\setminus\varphi(v)$.  If $\pbar(v)=\{\alpha\}$ is a singleton for some $\alpha\in [1,k]$, we  also write $\pbar(v)=\alpha$. 
For  $X\subseteq V(G)$,  let
$
\pbar(X)=\bigcup _{x\in X} \pbar(x).
$
The set $X$ is  \emph{$\varphi$-elementary} if $\pbar(x)\cap \pbar(y)=\emptyset$
for any distinct  $x,y\in X$.  

An edge
$e\in E(G)$ is a \emph{critical edge} of $G$ if $\chi'(G-e)<\chi'(G)$, and  
$G$ is  {\it edge $\Delta$-critical} or simply  \emph{$\Delta$-critical}  if $G$ is connected,  $\chi'(G)=\Delta+1$,  and every of its edge is critical. 
 The following result by Hilton and Zhao in~\cite{MR1172373} reveals certain properties of 
 an HZ graph. 
\begin{LEM}\label{biregular}
	If $G$ is  an HZ-graph with maximum degree $\Delta$, then  the following  holds.
	\begin{enumerate}[(a)]
		\item $G$ is $\Delta$-critical and $G_\Delta$ is 2-regular. 
		\item $\delta(G)=\Delta-1$, or $\Delta=2$ and $G$ is an odd cycle. 
		\item Every vertex of $G$ has at least two neighbors in $G_\Delta$. 
	\end{enumerate}
\end{LEM}

Let $\Delta\ge4$ and let $\Of_\Delta$ be the set of all graphs obtained  
from two graphs $H_1$ and $H_2$ by adding all  edges between $V(H_1)$ and  $V(H_2)$, where   $H_1$ is any 2-regular graph on $n_1$ vertices, 
$H_2$ is any $(\Delta-1-n_1)$-regular graph on $(\Delta-2)$ vertices, and   $n_1 \in [3,\Delta-1]$ such that $n_1+(\Delta-2)$ is odd.   
Stiebitz et al.  showed that 
Conjecture~\ref{Core Conjecture} is equivalent to the conjecture below. 
\begin{CON}[{\cite[Conjecture 4.10]{StiebSTF-Book}}]\label{core-conjecture-2}
	If $G$ is an HZ-graph with maximum degree $\Delta$,  then 
	either $G\in \Of_\Delta$, or 
	$\Delta=2$ and $G$ is an odd cycle, or $\Delta=3$ and $G=P^*$. 
\end{CON}

We will  prove this equivalent form of the Core Conjecture for $\Delta\ge 4$  by applying the following results.

\begin{THM}\label{Thm:vizing-fan}
	If $G$ is an HZ-graph with maximum degree $\Delta\ge 4$, then the following two statements hold.
	\begin{enumerate}[(i)]
		\item For any two  adjacent vertices  $u,v\in V_{\Delta}$, $N_{\Delta-1}(u)=N_{\Delta-1}(v)$.  \label{common}
		\item For any $r\in V_\Delta$, there exist  $s\in N_{\Delta-1}(r)$
		and  
		$\varphi\in \CC^\Delta(G-rs)$ such that $N_{\Delta-1}[r]$ 
		is $\varphi$-elementary.  
		\label{ele}
	\end{enumerate}
\end{THM}

For an HZ-graph $G$ with maximum degree $\Delta\ge 4$,  each component of $G_\Delta$ is a cycle by Lemma~\ref{biregular}. So Theorem~\ref{Thm:vizing-fan}~(\ref{common}) implies that $N_{\Delta-1}(x)=N_{\Delta-1}(y)$ for any two  vertices $x,y$   from the same cycle of $G_\Delta$. 

%

\begin{THM}
	\label{Thm:adj_small_vertex}
	If $G$ is an HZ-graph with maximum degree $\Delta\ge 4$,  then for any two  adjacent vertices $x, y\in V_{\Delta-1}$, $N_\Delta(x)=N_\Delta(y)$.
\end{THM}


\begin{THM}
	\label{Thm:nonadj_Delta_vertex}
	Let $G$ be an HZ-graph with maximum degree $\Delta\ge 7$ and $u, r\in V_{\Delta}$.
	If $N_{\Delta-1}(u)\ne N_{\Delta-1}(r)$ and $N_{\Delta-1}(u)\cap  N_{\Delta-1}(r)\ne \emptyset$,
	then $|N_{\Delta-1}(u)\cap  N_{\Delta-1}(r)|=\Delta-3$, i.e. $|N_{\Delta-1}(u)\setminus N_{\Delta-1}(r)|=|N_{\Delta-1}(r)\setminus  N_{\Delta-1}(u)|=1$. 
\end{THM}


\begin{COR}
	\label{Thm:adj_small_vertex2}
	If $G$ is an HZ-graph with maximum degree $\Delta\ge 7$  and there exist $u, v\in V_{\Delta}$ such that  $N_{\Delta-1}(u)\ne N_{\Delta-1}(v)$, 
	then $V_{\Delta-1}$ is an independent set in $G$.
\end{COR}
\pf
Assume to the contrary that there exist $x,y\in V_{\Delta-1}$ such that $xy\in E(G)$. By Lemma~\ref{biregular}, there exists $w\in N_\Delta(x)$.  By the assumption that there exist $u,v\in V_\Delta$ such that $N_{\Delta-1}(u)\ne N_{\Delta-1}(v)$,  there exists some $w'\in V_\Delta$ such that $N_{\Delta-1}(w)\ne N_{\Delta-1}(w')$.
We may further assume that the distance between $w$
and $w'$ in $G$ is shortest among all pairs of vertices $w_1$
and $w_1'$ such that $w_1\in N_\Delta(x)$ and $N_{\Delta-1}(w_1)\ne N_{\Delta-1}(w_1')$. 
 We claim that  $N_{\Delta-1}(w)\cap N_{\Delta-1}(w')\ne\emptyset$.  Let $P$ be a shortest path 
connecting $w$ and $w'$ in $G$. By the choice of $w$
and $w'$, $(V(P)\cap V_\Delta)\setminus \{w,w'\}$ contains no vertex $w^*$ such that $N_{\Delta-1}(w^*)=N_{\Delta-1}(w)$. 
Consequently, $V_\Delta\cap V(P)=\{w,w'\}$. 
Since $N_{\Delta-1}(w)\ne N_{\Delta-1}(w')$, it follows that 
$w$ and $w'$ are not on the same cycle of $G_\Delta$ and so $ww'\not\in E(G)$ by  Theorem~\ref{Thm:vizing-fan} (i). Thus $P-\{w,w'\}$ has at least one vertex. By Theorem~\ref{Thm:adj_small_vertex}, all vertices of  $P-\{w,w'\}$ have in $G$ the same set of neighbors from $V_\Delta$. Thus, both $w$ and $w'$
are $\Delta$-neighbors of each vertex from $P-\{w,w'\}$
and so 
 $N_{\Delta-1}(w)\cap N_{\Delta-1}(w')\ne\emptyset$. 
By Theorem~\ref{Thm:nonadj_Delta_vertex}, we  have 
$
|N_{\Delta-1}(w)\cap N_{\Delta-1}(w')|=\Delta-3,   
$
which together with Theorem~\ref{Thm:adj_small_vertex} implies
$
x,y\in N_{\Delta-1}(w)\cap N_{\Delta-1}(w').
$

Let	$N_{\Delta-1}(w')\setminus N_{\Delta-1}(w)=\{z\}$. We claim that $N_{\Delta-1}(z)=\emptyset$. For otherwise, let $z'\in N_{\Delta-1}(z)$. Clearly $z'\not=z$. By Theorem~\ref{Thm:adj_small_vertex}, $z'\in N_{\Delta-1}(w')\setminus N_{\Delta-1}(w)$, giving a contradiction to $N_{\Delta-1}(w')\setminus N_{\Delta-1}(w)=\{z\}$. We then claim that $N_\Delta(z)\subseteq N_\Delta(x)$. For otherwise let $w^*\in N_\Delta(z)\setminus N_\Delta(x)$. As $x\in N_{\Delta-1}(w')$ and $x\not\in N_{\Delta-1}(w^*)$, it follows that $w^*\ne w'$. Since $z\in N_{\Delta-1}(w^*)\cap N_{\Delta-1}(w')$, it follows that $|N_{\Delta-1}(w^*)\cap N_{\Delta-1}(w')|\ge \Delta-3$ by Theorem~\ref{Thm:nonadj_Delta_vertex} (it can happen that $N_{\Delta-1}(w^*)=N_{\Delta-1}(w')$). Thus $N_{\Delta-1}(w^*)\cap N_{\Delta-1}(w')$ contains at least one of $x$ and $y$ as $x,y\in  N_{\Delta-1}(w')$. As $xy\in E(G)$, we have $x,y\in N_{\Delta-1}(w^*)\cap N_{\Delta-1}(w')$  by Theorem~\ref{Thm:adj_small_vertex}. This  gives a contradiction to the choice of $w^*$. Therefore we have $N_{\Delta-1}(z)=\emptyset$ and $N_\Delta(z)\subseteq N_\Delta(x)$.  However,  $d_G(z)\le |N_\Delta(x)|<|N_\Delta(x)\cup \{y\}|\le d_G(x)$, contradicting $d_G(x)=d_G(z)=\Delta-1$. This completes the proof. \qed

We now prove Conjecture~\ref{core-conjecture-2} for $\Delta\ge 4$ as below. 

\begin{THM}\label{mainthm2}
	If $G$ is an HZ-graph with maximum degree $\Delta\ge 4$, then 
	$G\in \Of_\Delta$. 
\end{THM}
\begin{proof}

	Assume to the contrary that there exists an HZ-graph $G$ with maximum degree $\Delta\ge 4$ such that $G\not\in \Of_\Delta$. Let $n=|V(G)|$.  First assume that $N_{\Delta-1}(u)=N_{\Delta-1}(v)$ for every pair $u,v\in V_\Delta$.   Then $V_\Delta$, $V_{\Delta-1}$ and the edges between them form a complete bipartite graph. Since  $G\not\in \Of_\Delta$, it follows that $n$ is even. Let $r\in V_\Delta$. The  assumption above also implies that  $N_{\Delta-1}(r)=V_{\Delta-1}$. By Theorem~\ref{Thm:vizing-fan} \eqref{ele}, there exist $s\in N_{\Delta-1}(r)=V_{\Delta-1}$ and 
	$\varphi\in \CC^\Delta(G-rs)$ such that   
	$N_{\Delta-1}[r]=V_{\Delta-1}\cup \{r\}$ is $\varphi$-elementary, which thereby implies that $V(G)$ is $\varphi$-elementary. Therefore, each color in $\pbar(N_{\Delta-1}[r])$ is missed  at exactly one vertex in $V(G)$, showing that $n$ is odd. 
	This is
	a contradiction.

	We now assume that there exist $u, v\in V_{\Delta}$ such that  $N_{\Delta-1}(u)\ne N_{\Delta-1}(v)$. We further assume that $N_{\Delta-1}(u)\cap N_{\Delta-1}(v) \ne \emptyset$ (using the same argument to find $u$ and $v$ as for finding $w$ and $w'$ in the proof of Corollary~\ref{Thm:adj_small_vertex2}). 
	By Theorem~\ref{Thm:vizing-fan} \eqref{common},  the cycle  $C_u$ containing $u$ and the cycle $C_v$ containing $v$
	from $G_\Delta$ are distinct. Let $w\in N_{\Delta-1}(u)\cap N_{\Delta-1}(v)$. 
	Then $d_G(w)\ge |V(C_u)|+|V(C_v)|\ge 6$
	by Theorem~\ref{Thm:vizing-fan}~\eqref{common}.  
	Thus $\Delta=d_G(w)+1\ge 7$.  Applying Corollary~\ref{Thm:adj_small_vertex2}, it follows that $V_{\Delta-1}$ is an independent set of $G$.
	
	Let $A\subseteq V_\Delta$ be the set of all vertices $a$ satisfying $N_{\Delta-1}(a)=N_{\Delta-1}(u)$, and let $B\subseteq V_\Delta$ be the set of all vertices $b$ satisfying $N_{\Delta-1}(b)\ne N_{\Delta-1}(u)$ and $N_{\Delta-1}(b)\cap N_{\Delta-1}(u)\ne \emptyset$. Clearly $u\in A$ and $v\in B$, so $A$ and $B$ are non-empty. Partition $B$ into non-empty subsets $B_1,B_2,\ldots,B_t$ such that for each $i\in[1,t]$, all vertices in $B_i$ have the same neighborhood in $V_{\Delta-1}$. By Theorem~\ref{Thm:vizing-fan}~\eqref{common}, each of $A,B_1,B_2,\ldots,B_t$ induces a union of disjoint cycles in $G_\Delta$. So $|A|\ge 3$ and $|B_i|\ge 3$ for each $i\in[1,t]$.
	
	Now we claim   $t\ge \Delta-2$. 
	Assume otherwise  $t\le \Delta-3$.  
	Since for each $i\in [1,t]$, $|N_{\Delta-1}(A)\setminus N_{\Delta-1}(B_i)|=1$ by Theorem~\ref{Thm:nonadj_Delta_vertex} and  $|N_{\Delta-1}(A)|=\Delta-2$, there exists $z\in N_{\Delta-1}(A)$  such that $z\not\in N_{\Delta-1}(A)\setminus N_{\Delta-1}(B_i)$ for each $i\in [1,t]$, or equivalently,   $z\in N_{\Delta-1}(A)\cap \left(\bigcap_{i=1}^t N_{\Delta-1}(B_i)\right)$. Let  $z'\in N_{\Delta-1}(A)\setminus N_{\Delta-1}(B_1)$. Then 
	\begin{eqnarray*}
		|A|+\sum\limits_{1\le i\le t}|B_i|= d_G(z)=d_G(z')\le |A|+ \sum\limits_{2\le i\le t}|B_i|, 
	\end{eqnarray*}
	achieving a contradiction. Hence $t\ge \Delta-2$. 
	
	We now achieve a contradiction to the assumption  $\Delta\ge 7$
	by counting the number of edges in $G$ between $N_{\Delta-1}(A)$ and $A\cup B$.
	Note that $|N_{\Delta-1}(A)|=\Delta-2$. 
	Since each vertex in $B$ has exactly $\Delta-3$ neighbors in $N_{\Delta-1}(A)$ and $|B_i| \ge 3$ for each $i\in [1,t]$, we have
	\[ |E_G (A\cup B,N_{\Delta-1}(A))| = 
	|A|(\Delta-2) + 
	|\cup_{i=1}^tB_i | (\Delta-3) \ge 3(\Delta-2)+3t(\Delta -3) \ge 3(\Delta -2)^2. 
	\]
	
	On the other hand, since $N_{\Delta-1}(A)$ is an independent set and every vertex in it has degree $\Delta -1$ in $G$, we  have 
	\[
	|E_G (A\cap B,N_{\Delta-1}(A))| = (\Delta-1) (\Delta-2). 
	\]
	Since $\Delta\ge 2$, solving $\Delta$ in $(\Delta-1)(\Delta -2) \ge 3(\Delta -2)^2$ gives  $ \Delta\le 2.5$, 
	achieving a  desired contradiction. 
\end{proof}

\section{Definitions and previous results}

In this section, we recall essential concepts from~\cite{HZI}
and list a number of results that we will use as lemmas in the proof of 
Theorems~\ref{Thm:vizing-fan} to~\ref{Thm:nonadj_Delta_vertex}.

 Let $G$ be a graph,  $e\in E(G)$, $\varphi\in \CC^k(G-e)$ for some $k\ge 0$, and let
 $\alpha,\beta \in [1,k]$. Each component of $G-e$
induced on edges colored by $\alpha$ or $\beta$ is either a 
path or an even cycle,  which is called an \emph{$(\alpha,\beta)$-chain} of $G-e$
with respect to $\varphi$. 
Interchanging  $\alpha$ and $\beta$
on an $(\alpha,\beta)$-chain $C$ of $G$ gives a new edge $k$-coloring, which is denoted by 
$\varphi/C$. 
This operation  is called a \emph{Kempe change}. 

For $x,y\in V(G)$, if $x$ and $y$
are contained in the same  $(\alpha,\beta)$-chain, we say $x$ 
and $y$ are \emph{$(\alpha,\beta)$-linked} with respect to $\varphi$.
Otherwise, they are \emph{$(\alpha,\beta)$-unlinked}. 
If an $(\alpha,\beta)$-chain  $P$ is a path with one endvertex as $x$, we also denote  it by $P_x(\alpha,\beta,\varphi)$ and just write $P_x(\alpha,\beta)$ if $\varphi$ is understood.  For a vertex $u$ and an edge $uv$ contained in $P_x(\alpha,\beta,\varphi)$, 
we write 
  {$\mathit {u\in P_x(\alpha,\beta, \varphi)}$} and  {$\mathit {uv\in P_x(\alpha,\beta, \varphi)}$}.  
If $u,v\in P_x(\alpha,\beta,\varphi)$ such that $u$ lies between $x$ and $v$ on $P$, 
then we say that $P_x(\alpha,\beta,\varphi)$ \emph{meets $u$ before $v$}.

Let   
$T$ be  an alternating sequence of vertices  and edges of  $G$. We denote by \emph{$V(T)$}  
the set of vertices  contained in $T$, and by  
\emph{$E(T)$}  the set of edges contained in $T$. We simply write $\pbar(T)$ for $\pbar(V(T))$. 
If $V(T)$ is $\varphi$-elementary,
then for a color  $\tau\in \pbar(T)$,  we denote by  $\mathit{\pbar^{-1}_T(\tau)}$ the  unique vertex  in $V(T)$ at which $\tau$ 
is missed.  A coloring $\varphi'\in \CC^k(G-e)$
is  \emph{$(T,\varphi)$-stable} if for every  $x\in V(T)$ and every $f\in E(T)$, it holds that $\pbar'(x)=\pbar(x)$ and  $\varphi'(f)=\varphi(f)$.  Clearly, $\varphi$ is 
$(T,\varphi)$-stable, and if $\varphi_1\in \CC^k(G-e)$ is $(T,\varphi)$-stable, and  $\varphi_2\in \CC^k(G-e)$ is $(T,\varphi_1)$-stable, then $\varphi_2$
is also $(T,\varphi)$-stable.  

\subsection{Multifan}
Let $G$ be a graph,  $rs_1\in E(G)$ and $\varphi\in \CC^k(G-rs_1)$ for some $k\ge 0$.`
A \emph{multifan} centered at $r$ with respect to $rs_1$ and $\varphi$
is a sequence $$F_\varphi(r,s_1:s_p):=(r, rs_1, s_1, rs_2, s_2, \ldots, rs_p, s_p)$$ with $p\geq 1$ consisting of  distinct vertices and edges such that 
for every edge $rs_i$ with $i\in [2,p]$,  there is a vertex $s_j$ with $j\in [1,i-1]$ satisfying
	$\varphi(rs_i)\in \pbar(s_j)$. 
The following result can be found in \cite[Theorem~2.1]{StiebSTF-Book}.

\begin{LEM}
	\label{thm:vizing-fan1}
	Let $G$ be a class 2 graph and $F_\varphi(r,s_1:s_p)$  be a multifan with respect to   $rs_1$ and  $\varphi\in \CC^\Delta(G-rs_1)$. Then  the following statements  hold. 
	\begin{enumerate}[(a)]
		\item $V(F)$ is $\varphi$-elementary. \label{thm:vizing-fan1a}
		\item For any $\alpha\in \pbar(r)$ and any  $\beta\in \pbar(s_i)$ with $i\in [1,p]$,  $r$ 
		and $s_i$ are $(\alpha,\beta)$-linked with respect to $\varphi$. \label{thm:vizing-fan1b}
	\end{enumerate}
\end{LEM}

	Let $F_\varphi(r,s_1:s_p)$  be a multifan.  We call $s_{\ell_1},s_{\ell_2}, \ldots, s_{\ell_h}$, a subsequence of $s_2, \ldots, s_p$, an  \emph{$\alpha$-inducing sequence} for some $\alpha\in[1,k]$ with respect to $\varphi$ and $F$ if 
	$
	\varphi(rs_{\ell_1})= \alpha\in \pbar(s_1)$ and  $\varphi(rs_{\ell_i})\in \pbar(s_{\ell_{i-1}})$ for each  $i\in [2,h].
	$ (By this definition, $(r, rs_1, s_1, rs_{\ell_1}, s_{\ell_1}, \ldots, rs_{\ell_h}, s_{\ell_h})$ is also a multifan with respect to $rs_1$ and $\varphi$.)
	A  color in $\pbar(s_{\ell_i})$ for any $i\in[1,h]$ is an \emph{$\alpha$-inducing color} and is \emph{induced by} $\alpha$.   For $\alpha_i\in \pbar(s_{\ell_i})$
	and $\alpha_j\in \pbar(s_{\ell_j})$ with $i<j$ and $i,j\in [1,h]$, we write {$\mathit \alpha_i \prec \alpha_j$}.  For convenience, $\alpha$ itself is  an $\alpha$-inducing color and is induced by $\alpha$, and $\alpha\prec  \beta$
	for any $\beta \in \pbar(s_{\ell_i})$ and any $i\in [1,h]$. An $\alpha$-inducing color $\beta$ is called a \emph{last $\alpha$-inducing color} if  there does not exist any $\alpha$-inducing color $\delta$ such that $\beta \prec \delta$.

By Lemma~\ref{thm:vizing-fan1} (a), each color in $\pbar(F)\setminus \pbar(r)$ is induced by a unique color in $\pbar(s_1)$. Also if $\alpha_1$ and $\alpha_2$ are two distinct colors in $\pbar(s_1)$, then an $\alpha_1$-inducing sequence is disjoint with an $\alpha_2$-inducing sequence. The following result is 
 a consequence of Lemma~\ref{thm:vizing-fan1} (a).
\begin{LEM}[ {\cite[Lemma 3.2]{HZI}}]
	\label{thm:vizing-fan2}
	Let $G$ be a class 2 graph and $F_\varphi(r,s_1:s_p)$  be a multifan with respect to  $rs_1$ and  $\varphi\in \CC^\Delta(G-rs_1)$. For any two colors $\delta, \lambda$ with $\delta\in \pbar(s_i)$ and $\lambda\in \pbar(s_j)$ for some distinct $i,j\in [1,p]$, the following statements  hold.
	\begin{enumerate}[(a)]
		\item If $\delta$ and $\lambda$ are induced by different colors from $\pbar(s_1)$, then $s_i$ and $s_j$ are $(\delta, \lambda)$-linked with respect to $\varphi$. 
		\label{thm:vizing-fan2-a}
		\item If $\delta$ and $\lambda$ are induced by the same color from $\pbar(s_1)$ such that $\delta\prec\lambda$ and $s_i$ and $s_j$ are $(\delta, \lambda)$-unlinked with respect to $\varphi$, 
		then $r\in P_{s_j}(\lambda, \delta, \varphi)$.  	\label{thm:vizing-fan2-b}
	\end{enumerate}
	
\end{LEM}

By Lemma~\ref{biregular} (a), every edge of an HZ graph is critical. 
For an HZ-graph $G$ with maximum degree $\Delta\ge 3$, we let $rs_1\in E(G)$ with $r\in V_\Delta$ and $s_1\in N_{\Delta-1}(r):=\{s_1,s_2,\dots, s_{\Delta-2}\}$, and $\varphi\in \CC^\Delta(G-rs_1)$. Then we call $(G,rs_1,\varphi)$ a \emph{coloring-triple}. 
As $\Delta$-degree vertices in a multifan do not miss any color, 
for  multifans in  HZ-graphs, we add a further requirement in its definition as follows 
and we use this new definition  in the remainder of this paper.
\begin{Ass}
	For multifans in an HZ-graph,	all of its vertices except the center  have degree $\Delta-1$.   
\end{Ass}

Let $(G,rs_1,\varphi)$ be a coloring-triple and $F:=F_\varphi(r,s_1:s_p)$ be a multifan. By its definition,  $|\pbar(s_1)|=2$, $|\pbar(s_i)|=1$ for each $i\in [2,p]$, and  so every color in $\pbar(F)\setminus \pbar(r)$ is induced by one of the two colors in $\pbar(s_1)$. 
We call $F$ a \emph{typical multifan}, denoted $F_\varphi(r, s_1:s_\alpha:s_\beta):=(r, rs_1, s_1, rs_2, s_2, \ldots,rs_\alpha, s_\alpha, rs_{\alpha+1}, s_{\alpha+1}, \ldots, rs_\beta, s_\beta)$, where $\beta:=p$,  if $\pbar(r)=1$ (recall we denote $\pbar(v)$ by a number if $|\pbar(v)|=1$), $\pbar(s_1)=\{2,\Delta\}$, and if $|V(F)|\ge 3$, then $\varphi(rs_{\alpha+1})=\Delta$ and $\pbar(s_{\alpha+1})=\alpha+2$ (if $\beta>\alpha$), and 
for each $i\in [2,\beta]$ with $i\ne \alpha+1$, $\varphi(rs_i)=i$  and $\pbar(s_i)=i+1$.	
It is clear that $s_2, \ldots, s_\alpha$ is the longest 
$2$-inducing sequence  and $ s_{\alpha+1}, \ldots, s_\beta$ is the longest 
$\Delta$-inducing sequence of $F_\varphi(r, s_1:s_\alpha:s_\beta)$. 
By relabelling   vertices and colors if necessary, any multifan in an HZ-graph can be assumed to be a typical multifan, see Figure~\ref{f1} (a) for a depiction.  If $\alpha=\beta$, then we write $F_\varphi(r,s_1:s_\alpha)$ for $F_\varphi(r, s_1:s_\alpha:s_\beta)$, and call it  a {\it typical 2-inducing multifan}.  

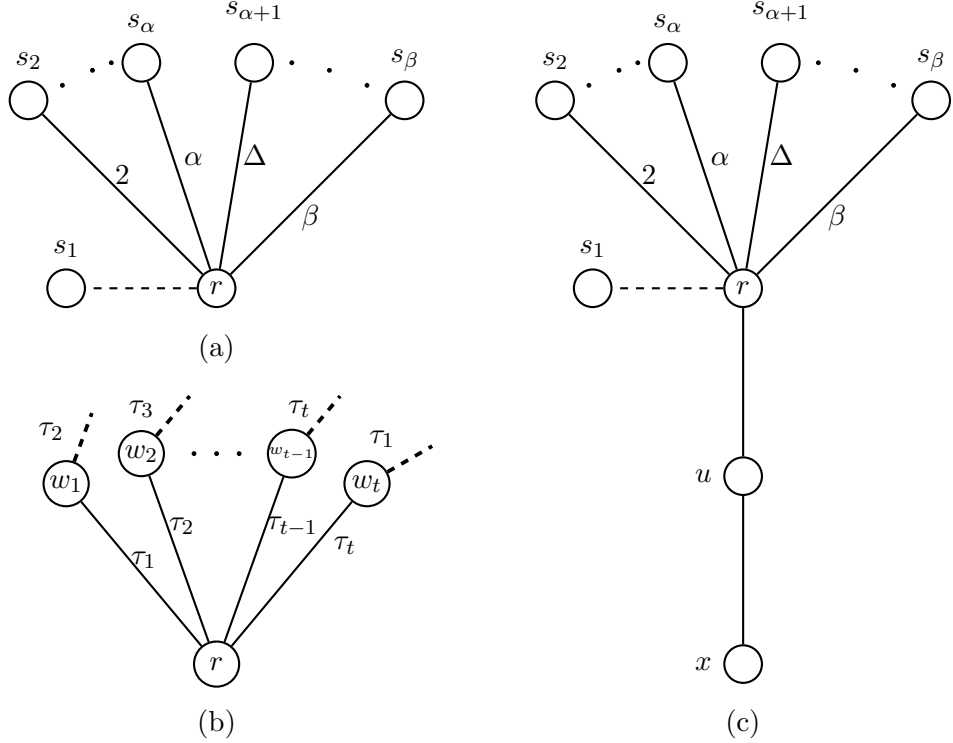
\begin{figure}[!htb]
	\begin{center}
		\begin{tikzpicture}[scale=1]
		
		{\tikzstyle{every node}=[draw ,circle,fill=white, minimum size=0.5cm,
			inner sep=0pt]
			\draw[black,thick](-1,0) node[label={below: }] (r)  {$r$};
			\draw[black,thick](-3,0) node[label={above: $s_1$}] (s1) {};
			\draw[black,thick](-3.5,2.5) node[label={above: $s_2$}] (s2) {};
			\draw[black,thick](-2,3) node[label={above: $s_\alpha$}] (sa) {};
			\draw[black,thick](-0.5,3) node[label={above: $s_{\alpha+1}$}] (sa1) {};
			\draw[black,thick](1.5,2.5) node[label={above: $s_\beta$}] (sb) {};
			
		}
		\path[draw,thick,black, dashed]
		(r) edge node[name=la,above,pos=0.5] {\color{black}} (s1);
		
		\path[draw,thick,black]
		(r) edge node[name=la,above,pos=0.5] {\color{black}$2$} (s2)
		(r) edge node[name=la,above,pos=0.5] {\color{black}\quad$\alpha$} (sa)
		(r) edge node[name=la,above,pos=0.5] {\color{black}\quad\,\,$\Delta$} (sa1)
		(r) edge node[name=la,below,pos=0.5] {\color{black}$\beta$} (sb);

		{\tikzstyle{every node}=[draw ,circle,fill=black, minimum size=0.05cm,
			inner sep=0pt]
			\draw(-3.05,2.7) node (f1)  {};
			\draw(-2.65,2.9) node (f1)  {};
			\draw(-2.4,3) node (f1)  {};
			\draw(0,3) node (f1)  {};
			\draw(0.5,2.9) node (f1)  {};
			\draw(1,2.7) node (f1)  {};
		} 
		\draw(-1,-0.8) node (f1)  {(a)};
		\begin{scope}[shift={(7,0)}]
		{\tikzstyle{every node}=[draw ,circle,fill=white, minimum size=0.5cm,
			inner sep=0pt]
			\draw[black,thick](-1,0) node[label={left: }] (r)  {$r$};
			\draw[black,thick](-3,0) node[label={above: $s_1$}] (s1) {};
			\draw[black,thick](-3.5,2.5) node[label={above: $s_2$}] (s2) {};
			\draw[black,thick](-2,3) node[label={above: $s_\alpha$}] (sa) {};
			\draw[black,thick](-0.5,3) node[label={above: $s_{\alpha+1}$}] (sa1) {};
			\draw[black,thick](1.5,2.5) node[label={above: $s_\beta$}] (sb) {};
			\draw[black,thick](-1,-2.5) node[label={left: $u$}] (u)  {};
			\draw[black,thick](-1,-5) node[label={left: $x$}] (x)  {};
		}
		\path[draw,thick,black, dashed]
		(r) edge node[name=la,above,pos=0.5] {\color{black}} (s1);
		
		\path[draw,thick,black]
		(r) edge node[name=la,above,pos=0.5] {\color{black}$2$} (s2)
		(r) edge node[name=la,above,pos=0.5] {\color{black}\quad$\alpha$} (sa)
		(r) edge node[name=la,above,pos=0.5] {\color{black}\quad\,\,$\Delta$} (sa1)
		(r) edge node[name=la,below,pos=0.5] {\color{black}$\beta$} (sb)
		(r) edge node[name=la,below,pos=0.5] {\color{black}} (u)
		(u) edge node[name=la,below,pos=0.5] {\color{black}} (x);
		
		{\tikzstyle{every node}=[draw ,circle,fill=black, minimum size=0.05cm,
			inner sep=0pt]
			\draw(-3.05,2.7) node (f1)  {};
			\draw(-2.65,2.9) node (f1)  {};
			\draw(-2.4,3) node (f1)  {};
			\draw(0,3) node (f1)  {};
			\draw(0.5,2.9) node (f1)  {};
			\draw(1,2.7) node (f1)  {};
		} 
		
		\draw(-1,-5.8) node (f1)  {(c)};
		
		\end{scope}	
		
		\begin{scope}[shift={(-1,-2)}]
		{\tikzstyle{every node}=[draw ,circle,fill=white, minimum size=0.6cm,
			inner sep=0pt]
			\draw[black,thick] (0, -3) node (r)  {$r$};
			\draw[black,thick] (-2, 0.4-1) node (sa)  {$w_1$};
			\draw [black,thick](-1, 0.8-1) node (sa2)  {$w_2$};
			\draw [black,thick](1, 0.8-1) node (sb)  {\tiny$w_{t-1}$};
			\draw [black,thick](2, 0.4-1) node (sb2)  {$w_t$};
		}
		\path[draw,thick,black]
		(r) edge node[name=la,pos=0.6] {\color{black}\quad$\tau_1$} (sa)
		(r) edge node[name=la,pos=0.7] {\color{black}\quad$\tau_2$} (sa2)
		(r) edge node[name=la,pos=0.7] {\color{black}\quad\,\,\,\,\,$\tau_{t-1}$} (sb)
		(r) edge node[name=la,pos=0.7] {\color{black}\qquad$\tau_t$} (sb2);

		\draw[dashed, black, line width=0.5mm] (sa)--++(70:1cm); 
		\draw[dashed, black, line width=0.5mm] (sa2)--++(50:1cm); 
		\draw[dashed, black, line width=0.5mm] (sb)--++(50:1cm); 
		\draw[dashed, black, line width=0.5mm] (sb2)--++(30:1cm);

		\draw[black] (-2.2, 1.1-1) node {$\tau_2$};  
		\draw[black] (-1.0, 1.4-1) node {$\tau_3$};  
		\draw[black] (1.1, 1.4-1) node {$\tau_{t}$}; 
		\draw[black] (2.2, 1.0-1) node {$\tau_{1}$}; 
		
		{\tikzstyle{every node}=[draw ,circle,fill=black, minimum size=0.05cm,
			inner sep=0pt]
			
			\draw(-0.3,0.8-1) node (f1)  {};
			\draw(0,0.8-1) node (f1)  {};
			\draw(0.3,0.8-1) node (f1)  {};

		} 
		
		\draw(0,-3.8) node (f1)  {(b)};
		\end{scope}
		\end{tikzpicture}
	
	\vspace{-0.5cm}
	\end{center}
	\caption{(a) A typical multifan $F_\varphi(r, s_1:s_\alpha:s_\beta)$, where $\pbar(r)=1$ and $\pbar(s_1)=\{2,\Delta\}$;  (b) A rotation centered at $r$, where a dashed line at a vertex indicates a color missing at the vertex; (c)  A lollipop centered at $r$, where $x$ can be the same as some $s_\ell$ for $\ell\in [\beta+1, \Delta-2]$.}
	\label{f1}
\end{figure}

\subsection{Kierstead path}

Let $G$ be a graph, $e=v_0v_1\in E(G)$, and  $\varphi\in \CC^k(G-e)$ for some integer $k\ge 0$.
A \emph{Kierstead path}  with respect to $e$ and $\varphi$
is a sequence $K=(v_0, v_0v_1, v_1, v_1v_2, v_2, \ldots, v_{p-1}, v_{p-1}v_p,  v_p)$ with $p\geq 1$ consisting of  distinct vertices and  edges such that for every edge $v_{i}v_{i+1}$ with $i\in [1,p-1]$,  there exists $j\in [0,i-1]$ satisfying 
	$\varphi(v_{i}v_{i+1})\in \pbar(v_j)$.

A Kierstead path with at most 3 vertices is a multifan. We consider Kierstead paths with $4$ vertices.
 Statement $(a)$ below was proved in Theorem 3.3 from~\cite{StiebSTF-Book} and statement $(b)$ is a consequence of $(a)$. 

\begin{LEM}[]\label{Lemma:kierstead path1}
	Let $G$ be a class 2 graph,
	$v_0v_1\in E(G)$, and $\varphi\in \CC^\Delta(G-v_0v_1)$. If $K=(v_0, v_0v_1, v_1, v_1v_2,  v_2, v_2v_3, v_3)$ is a Kierstead path with respect to $v_0v_1$
	and $\varphi$, then the following statements hold.
	\begin{enumerate}[(a)]
		\item If $\min\{d_G(v_1), d_G(v_2)\}<\Delta$, then $V(K)$ is $\varphi$-elementary.
		\item For any two colors $\alpha,\delta$ with $\alpha\in\pbar(v_0)$ and 
		$\delta\in \pbar(v_3)$, if  $\min\{d_G(v_1), d_G(v_2)\}<\Delta$ and $\alpha \not\in \{\varphi(v_1v_2), \varphi(v_2v_3)\}$, 
		then 
		$v_3$ and $v_0$ are $(\alpha,\delta)$-linked with respect to $\varphi$.
	\end{enumerate}

\end{LEM}

\subsection{Pseudo-multifan}
Let $G$ be a graph,  $rs_1\in E(G)$ and $\varphi\in \CC^k(G-rs_1)$ for some $k\ge 0$.
A multifan $F_\varphi(r,s_1:s_p)$ is  \emph{maximum} at $r$ if $|V(F)|$ is maximum among all multifans with respect to $rs$ for all $s\in N_G(r)$  and all $\varphi'\in \CC^k(G-rs)$. 
A  \emph{pseudo-multifan}  with respect to $rs_1$ and $\varphi$ is an alternating sequence
  $S:=S_\varphi(r,s_1:s_t:s_p):=(r, rs_1, s_1, rs_2, s_2, \ldots,rs_t, s_t, rs_{t+1},  s_{t+1}, \ldots, s_{p-1},  rs_p, s_p)$ 
with $t,p \ge 1$  of distinct vertices and edges  satisfying the following conditions:
\begin{enumerate}[(P1)]
	\item the subsequence $F:=(r, rs_1, s_1, rs_2, s_2, \ldots,rs_t, s_t)$ is a maximum multifan at $r$.
	\item $V(S)$ is $\varphi'$-elementary  for every $(F,\varphi)$-stable $\varphi'\in \CC^k(G-rs_1)$.
\end{enumerate}

Every maximum multifan is a pseudo-multifan, and if $S$ is a pseudo-multifan 
with respect to  $\varphi$ and a multifan  $F$,  then by the definition above,  $S$ is a pseudo-multifan under every  $(F,\varphi)$-stable coloring $\varphi'$. We call a pseudo-multifan  $S$ {\it typical} (resp. {\it typical 2-inducing}) if the maximum multifan that is contained in $S$
is typical (resp. typical 2-inducing).

Let $(G,rs_1,\varphi)$ be a coloring-triple and $i,j\in [2,\Delta-2]$.  
The \emph{shift from $s_i$ to $s_j$}  is an operation that,  for each $\ell $ with $ \ell\in[i,j]$,  recolor $rs_\ell$ by  the color in $\pbar(s_\ell)$.  We will apply a shift either on a sequence of vertices from a multifan 
or on a rotation.

%

\begin{LEM}\label{2-inducing}
	Let $(G,rs_1,\varphi)$ be a coloring-triple. Then for every typical pseudo-multifan $S:=S_\varphi(r, s_1: s_t: s_{p})$, there exists a coloring $\varphi'\in \CC^\Delta(G-rs_t)$ and a pseudo-multifan $S^*$ centered at $r$ with respect to $rs_t$ and $\varphi'$ such that $V(S^*)=V(S)$ and $S^*$ is typical 2-inducing. 
\end{LEM}
\pf Let  $F=F_\varphi(r,s_1:s_\alpha:s_\beta)$ be the typical multifan contained in $S$, where $s_\beta=s_t$. 
If $\beta=\alpha$, then we are done. Thus we assume   $\beta\ge \alpha+1 \ge 3$. Let $\varphi'$ be obtained from $\varphi$ by uncoloring $rs_\beta$, shift from $s_{\alpha+1}$ to $s_{\beta-1}$ and coloring $rs_1$ by $\Delta$. Now $\pbar'(s_\beta)=\{\beta,\beta+1\}$, $F^*=(r, rs_\beta, s_\beta, rs_{\beta-1}, s_{\beta-1}, \ldots,  rs_{\alpha+1}, s_{\alpha+1}, rs_1, s_1,  \ldots, rs_\alpha, s_\alpha)$ is a $\beta$-inducing multifan with respect to $rs_\beta$ and $\varphi'$.

We next show that $S^*=(F^*,rs_{t+1}, s_{t+1}, \ldots, rs_{p},s_{p})$
is a pseudo-multifan with respect to $rs_\beta$ and $\varphi'$. 
 Since $|V(F^*)|=|V(F)|$, $F^*$ is also a maximum multifan at $r$. Thus  it suffices to show that for any $(F^*,\varphi')$-stable $\varphi''\in \CC^\Delta(G-rs_\beta)$, $V(S^*)$ is $\varphi''$-elementary. 
Suppose to  the contrary that there exists $(F^*,\varphi')$-stable $\varphi''\in \CC^\Delta(G-rs_\beta)$ but $V(S^*)$
is not $\varphi''$-elementary.  As $\varphi''$ is $(F^*,\varphi')$-stable,
we can undo the operations we did before. Specifically, let $\varphi'''$ be the coloring obtained from $\varphi''$ by uncoloring $rs_1$, shift from $s_{\alpha+1}$ to $s_{\beta-1}$ and coloring $rs_{\beta}$ by $\beta$. Then $\varphi'''$ is $(F,\varphi)$-stable and $\pbar'''(S^*)=\pbar''(S^*)$.
Thus, $V(S^*)$ is  not $\varphi''$-elementary implies that   $V(S^*)$ is not $\varphi'''$-elementary.  Since $V(S^*)=V(S)$, this contradicts the assumption that $V(S)$
is elementary under any  $(F,\varphi)$-stable coloring. Therefore, $S^*$ is a pseudo-multifan with respect to $rs_\beta$ and $\varphi'$.
By renaming  colors
and vertices, we can assume that $F^*$ is typical 2-inducing and so $S^*$ is typical 2-inducing. 
\qed

Let $(G,rs_1,\varphi)$ be a coloring-triple. 
A sequence of distinct vertices 
$w_1,  \ldots, w_t \in N_{\Delta-1}(r)$  form a \emph{rotation}  if $\{w_1,\ldots, w_t\}$ is $\varphi$-elementary, and 
for each $\ell$ with $\ell\in [1,t]$, it holds that  $\varphi(rw_\ell)=\pbar(w_{\ell-1})$, where $w_0:=w_t$. 
An example of a rotation is given in Figure~\ref{f1} (b). 

\begin{LEM}[ {\cite[Theorem 2.5]{HZI}}]\label{pseudo-fan-ele}
	Let  $(G,rs_1,\varphi)$ be a coloring-triple, $S:=S_\varphi(r, s_1: s_t: s_{\Delta-2})$  be a pseudo-multifan 
	with $F:=F_\varphi(r,s_1:s_t)$ being  the maximum multifan contained in it. Let	$j\in [t+1,\Delta-2]$ and $\delta \in \pbar(s_j)$.   Then 
	\begin{enumerate}[(a)]
		\item $\{s_{t+1}, \ldots, s_{\Delta-2}\}$ can be partitioned 
		into rotations with respect to $\varphi$. \label{pseudo-a}

		\item $s_j$ and $r$ are $(1,\delta)$-linked with respect to  $\varphi$ . \label{pseudo-a1}
		\item For every color $\gamma\in \pbar(F)$ with $\gamma \ne 1$,  it holds that $r\in P_{y}(\gamma,\delta)=P_{s_j}(\gamma,\delta)$, where $y=\mathit{\pbar_{F}^{-1}(\gamma)}$.  
		Furthermore, for $z\in N_G(r)$ such that $\varphi(rz)=\gamma$,  
		$P_{y}(\gamma,\delta)$	meets $z$ before  $r$.  \label{pseudo-b}
		\item For every $\delta^*\in \pbar(S)\setminus \pbar(F)$ with $\delta^*\ne \delta$, it holds that $P_{y}(\delta,\delta^*)=P_{s_{j}}(\delta,\delta^*)$,  where $y=\mathit{\pbar_{S}^{-1}(\delta^*)}$. Furthermore, either $r\in P_{s_j}(\delta,\delta^*)$ or $P_r(\delta, \delta^*)$ is an even cycle. 
		\label{pseudo-c}
	\end{enumerate}
\end{LEM}

\subsection{Lollipop}

If $F=(a_1,\ldots, a_t )$ is a sequence, then for a new 
entry $b$, 
$(F, b)$  denotes the sequence $(a_1,\ldots, a_t, b)$. 
Let $(G,rs_1,\varphi)$ be a coloring-triple. 
A \emph{lollipop}
centered at $r$  is a sequence 
$L=(F, ru, u,ux, x)$ of distinct vertices and edges such that $F=F_\varphi(r,s_1:s_\alpha:s_\beta)$ is a typical multifan, $u\in N_\Delta(r)$ and $x\in N_{\Delta-1}(u)$ 
with $x\not\in\{s_1,\ldots, s_\beta\}$ (see Figure~\ref{f1} (c) for a depiction).

\begin{LEM}[{\cite[Lemma 5.1]{HZI}}]\label{Lemma:extended multifan}
	Let  $(G,rs_1,\varphi)$ be a coloring-triple, 
	$F:=F_\varphi(r,s_1:s_\alpha:s_\beta)$ be a typical multifan, and  $L:=(F,ru,u,ux,x)$ be a lollipop centered at $r$ such that  $\varphi(ru)=\alpha+1$ and $\pbar(x)=\alpha+1$.  Then  
	\begin{enumerate}[(a)]
		\item  $\varphi(ux)\ne 1$ and  $ux\in P_r(1,\varphi(ux))$.  \label{Evizingfan-a}
		
		If $\varphi(ux)=\tau$ 
		is a 2-inducing color with respect to $\varphi$ and $F$, then the following  holds. 
		
		\item Let $P_x(1,\tau)$ be the $(1,\tau)$-chain starting at $x$ in $G-rs_1-ux$. Then $P_x(1,\tau)$ ends at $r$.  \label{Evizingfan-b}
		\item For any 2-inducing color $\delta$  of $F$ with $\tau\prec \delta$, 
		we have $r\in P_{s_1}(\delta,\Delta)=P_{s_{\delta-1}}(\delta,\Delta)$. \label{Evizingfan-c}
		\item  For any $\Delta$-inducing color $\delta$ of $F$, we have $r\in P_{s_{\delta-1}}(\alpha+1,\delta)=P_{s_{\alpha}}(\alpha+1,\delta)$, where $s_{\Delta-1}=s_1$ if $\delta=\Delta$.
		\label{Evizingfan-d}
		\item For any 2-inducing color $\delta$  of $F$ with $\delta\prec \tau$,  we have 
		$r\in P_{s_\alpha}(\delta, \alpha+1)=P_{s_{\delta-1}}(\delta, \alpha+1)$. \label{Evizingfan-e}
	\end{enumerate}
\end{LEM}

Let  $(G,rs_1,\varphi)$ be a coloring-triple. 
For a color $\alpha \in [1,\Delta]$, a sequence of 
{\it Kempe  $(\alpha,*)$-changes}  is a sequence of  
Kempe changes that each involve the exchanging of the color $\alpha$
with another color from $[1,\Delta]$.   

\begin{LEM}[{\cite[Lemma 5.2]{HZI}}]\label{Lemma:pseudo-fan0}
	Let  $(G,rs_1,\varphi)$ be a coloring-triple, 
	$F:=F_\varphi(r,s_1:s_\alpha:s_\beta)$ be a typical multifan, and  $L:=(F,ru,u,ux,x)$ be a lollipop centered at $r$ such that  $\varphi(ru)=\alpha+1$.   Then for   $w_1\in \{s_{\beta+1}, \ldots, s_{\Delta-2} \}$ with $\varphi(rw_1)=\tau_1\in [\beta+2, \Delta-1]$,  the following statements hold.
	\begin{enumerate}[(1)] 
		\item \label{only-fan}	
		If exists a vertex $w\in V(G)\setminus (V(F)\cup \{w_1\})$ such that $w\in P_r(1,\tau_1,\varphi')$
		for every {\bf $(F,\varphi)$-stable} $\varphi' \in \CC^{\Delta}(G-rs_1)$ with $\varphi'(ru)=\alpha+1$,   
		then there exists a sequence of distinct vertices $w_1, \ldots, w_t\in  \{s_{\beta+1}, \ldots, s_{\Delta-2} \}$ satisfying the following conditions: 
		\begin{enumerate}[(a)]
			\item  $\varphi(rw_{i+1})=\pbar(w_i)\in  [\beta+2, \Delta-1]$ for each $i\in [1,t-1]$;
			\label{Lemma:pseudo-fan0-a}
			\item  $r$ and $w_i$  are $(1,\pbar(w_i))$-linked with respect to $\varphi$ for each $i\in[1,t]$; \label{Lemma:pseudo-fan0-b}
			\item $\pbar(w_t)=\tau_1$. \label{Lemma:pseudo-fan0-c}
		\end{enumerate}
		\item \label{fan-and-x}
		If  $\pbar(x)=\alpha+1$ and there exists a vertex $w\in V(G)\setminus (V(F)\cup \{w_1\})$ such that $w\in P_r(1,\tau_1,\varphi')$
		for every  {\bf $(L,\varphi)$-stable} $\varphi' \in \CC^{\Delta}(G-rs_1)$  obtained from $\varphi$ 
		through a sequence of Kempe  $(1,*)$-changes  not using  $r$
		or $x$ as endvertices,  
		then there exists  a sequence of distinct vertices $w_1, \ldots, w_t\in  \{s_{\beta+1}, \ldots, s_{\Delta-2} \}$ satisfying the following conditions: 
		\begin{enumerate}[(a)]
			\item  $\varphi(rw_{i+1})=\pbar(rw_i)\in  [\beta+2, \Delta-1]$  for each $i\in [1,t-1]$;
			\label{Lemma:pseudo-fan0-a1}
			\item $r$ and $w_i$   are $(1,\pbar(w_i))$-linked with respect to $\varphi$ for each $i\in[1,t-1]$; \label{Lemma:pseudo-fan0-b1}
			\item $\pbar(w_t)=\tau_1$ or $\pbar(w_t)=\alpha+1$.  If $\pbar(w_t)=\tau_1$,
			then $w_t$ and $r$ are $(1,\tau_{1})$-linked with respect to $\varphi$. \label{Lemma:pseudo-fan0-c1}
		\end{enumerate}
	\end{enumerate}
\end{LEM} 

By the definition, the sequence $w_1, \ldots, w_t$ in Lemma~\ref{Lemma:pseudo-fan0}~\eqref{only-fan} and in the case of  Lemma~\ref{Lemma:pseudo-fan0}~\eqref{fan-and-x} when $\pbar(w_t)=\tau_1$ form a rotation
with the additional property that $\pbar(w_i) \in [\beta+2,\Delta-1]$ and $r$ and $w_i$ are $(1,\pbar(w_i))$-linked for each $i\in[1,t]$. We call such a rotation a \emph{stable rotation}. 
In the case of  Lemma~\ref{Lemma:pseudo-fan0}~\eqref{fan-and-x} when $\pbar(w_t)=\alpha+1$,
we call $w_1, \ldots, w_t$ a \emph{near stable rotation}. 
For $u,v\in V(G)$, we write $u\sim v$ if $u$ and $v$ are adjacent in $G$, and write $u\not\sim v$ otherwise.

\begin{LEM}[{\cite[Corollary 2.7]{HZI}}]\label{Lem:2-non-adj1}
	Let  $(G,rs_1,\varphi)$ be a coloring-triple,
	$F:=F_\varphi(r,s_1:s_\alpha)$ be a typical 2-inducing  multifan, and  $L:=(F,ru,u,ux,x)$ be a lollipop centered at $r$.  If $\varphi(ru)=\alpha+1$, $\pbar(x)=\alpha+1$, 
	and $\varphi(ux)=\Delta$, 
	then  $u\not\sim s_1$ and $u\not\sim s_\alpha$. 
\end{LEM}

\begin{LEM}[ {\cite[Theorem 2.8]{HZI}}]\label{Lem:2-non-adj2}
	Let  $(G,rs_1,\varphi)$ be a coloring-triple, 
	$F:=F_\varphi(r,s_1:s_\alpha)$ be a typical 2-inducing  multifan, and  $L:=(F,ru,u,ux,x)$ be a lollipop centered at $r$.  If $\varphi(ru)=\alpha+1$, $\pbar(x)=\alpha+1$, 
	and $\varphi(ux)=\mu\in \pbar(F)$ is a 2-inducing color of $F$, 
	then  $u\not\sim s_{\mu-1}$ and $u\not\sim s_\mu$.  
\end{LEM}

Let $G$ be a graph,  $rs_1\in E(G)$ and $\varphi\in \CC^\Delta(G-rs_1)$.
Let $\alpha,\beta,\gamma,\tau\in [1,\Delta] $ and $x,y\in V(G)$. 
If $P$ is an 
$(\alpha,\beta)$-chain  containing both $x$ and $y$ such that $P$ is a path,  
we denote by $\mathit{P_{[x,y]}(\alpha,\beta, \varphi)}$  the subchain  of $P$ that has endvertices $x$
and $y$.  

Suppose  $|\pbar(x)\cap \{\alpha,\beta\}|=1$. Then an \emph{$(\alpha,\beta)$-swap} 
at $x$ is just the Kempe change on $P_x(\alpha,\beta,\varphi)$. By convention, an	$(\alpha,\alpha)$-swap at $x$ does  nothing at $x$. If also $|\pbar(y)\cap \{\alpha,\beta\}|=1$, then an 
$(\alpha,\beta)$-swap \emph{at both $x$ and $y$} is the Kempe change on $P_x(\alpha,\beta,\varphi)$ if $x$ 
and $y$ are $(\alpha,\beta)$-linked with respect to $\varphi$, and is obtained 
from $\varphi$ by first doing an $(\alpha,\beta)$-swap at $x$
and then doing an $(\alpha,\beta)$-swap at $y$ if $x$ 
and $y$ are $(\alpha,\beta)$-unlinked with respect to $\varphi$. 
Suppose $\beta_0\in \pbar(x)$
and $\beta_1,\ldots \beta_t\in \varphi(x)$ for colors $\beta_0, \ldots, \beta_t\in [1,\Delta]$
for some integer $t \ge 1$. Then a 
$$
(\beta_0,\beta_1)-(\beta_1,\beta_2)-\ldots-(\beta_{t-1},\beta_t)-\text{swap}
$$  
at $x$ consists of $t$ Kempe changes: let $\varphi_0=\varphi$,
then $\varphi_i=\varphi_{i-1}/P_x(\beta_{i-1},\beta_i, \varphi_{i-1})$ for each $i\in [1,t]$.  
Suppose the current color of an  edge $uv$ of $G$
is $\alpha$, the notation  $\mathit{uv: \alpha\rightarrow \beta}$  means to recolor  the edge  $uv$ using the color $\beta$. 

We will use a  matrix with two rows to denote a sequence of coloring operations  taken based on $\varphi$.
For example,   the matrix below indicates 
three operations taken on the graph:
\[
\begin{bmatrix}
P_{[a, b]}(\alpha, \beta,\varphi)  & s_c:s_{d} & rs\\
\alpha/\beta & \text{shift} & \gamma \rightarrow \tau 
\end{bmatrix}.
\]
\begin{enumerate}[Step 1]
	\item Exchange $\alpha$ and $\beta$ on the $(\alpha,\beta)$-subchain $P_{[a, b]}(\alpha, \beta,\varphi)$.
	\item Based on the coloring obtained from  Step 1, shift from $s_c$ to $s_d$
	for vertices $s_c, \ldots, s_d$. 
	
	\item Based on the coloring obtained from  Step 2,  do  $rs: \gamma \rightarrow \tau $. 
\end{enumerate}

In the reminder,  for simpler description,  we may skip the phrase ``with respect to $\varphi$'' in related notation, which then needs to be understood with respect to the current edge coloring. 

\section{Proof of Theorem~\ref{Thm:vizing-fan}}\label{thm2.1proof}
  We prove the following version of Theorem \ref{Thm:vizing-fan}.

\begin{THM}\label{Thm:vizing-fan2b}
	If $G$ is an HZ-graph with maximum degree $\Delta\ge 4$, then for every vertex $r\in V_{\Delta}$, the following two statements hold. 
	\begin{enumerate}[(i)]
		\item For every $u\in N_{\Delta}(r)$, 
		$N_{\Delta-1}(r)=N_{\Delta-1}(u)$.  \label{common2}
		\item There exist $s_1\in N_{\Delta-1}(r)$ 
		and  
		$\varphi\in \CC^\Delta(G-rs_1)$ such that $N_{\Delta-1}[r]$ is the vertex set of  a typical 2-inducing pseudo-multifan with respect to $rs_1$ and $\varphi$.   Consequently $N_{\Delta-1}[r]$ 
		is $\varphi$-elementary.  
		\label{ele2}
	\end{enumerate}
\end{THM}
\begin{proof}
	Let  $N_{\Delta-1}(r)=\{s_1,\ldots, s_{\Delta-2}\}$. 
	We choose a vertex in $N_{\Delta-1}(r)$, say $s_1$, a coloring $\varphi\in\CC^\Delta(G-rs_1)$ and a multifan $F$  with respect to $rs_1$
	and $\varphi$ such that $F$ is maximum at $r$. 
	That is, $|V(F)|$ is maximum among all multifans with respect to $rs_i$ for any $i\in[1,\Delta-2]$ and any $\varphi'\in \CC^\Delta(G-rs_i)$.
	Assume that $\pbar(r)=1$ and $\pbar(s_1)=\{2,\Delta\}$, and $F=F_\varphi(r,s_1:s_p)$ 
	is such a multifan. Furthermore, by relabeling vertices and colors, we assume that $F$ is typical. 
	As a maximum multifan at $r$ is itself a pseudo-multifan, by Lemma~\ref{2-inducing}, we  assume that 
	$F_\varphi(r,s_1:s_p)=F_\varphi(r,s_1:s_\alpha)$ is a typical  2-inducing 
	multifan, where $\alpha=p$.  
	
	Let $u\in N_{\Delta}(r)$ and assume $N_{\Delta-1}(r) \ne N_{\Delta-1}(u)$. Roughly speaking,  the main proof idea is the following. 
	By assuming $\varphi(ru)=\alpha+1$ and $\pbar(x)=\alpha+1$
	for $x\in N_{\Delta-1}(u)\setminus N_{\Delta-1}(r)$, we will apply Lemmas~\ref{Lem:2-non-adj1} and \ref{Lem:2-non-adj2} to show that $u$ has at least two $(\Delta-1)$-neighbors outside
	of $N_{\Delta-1}(r)$. By further applying Lemmas~\ref{Lem:2-non-adj1} and \ref{Lem:2-non-adj2}, 
	we can even find three $(\Delta-1)$-neighbors of $u$ outside
	of $N_{\Delta-1}(r)	$.  A contradiction is then deduced at that point.

	\begin{CLA}\label{ux-color}
			We may  assume that $\varphi(ru)=\alpha+1$, which is the last $2$-inducing color of 
		$F_\varphi(r,s_1:s_\alpha)$. 
	\end{CLA}
	\proof[Proof of Claim~\ref{ux-color}]
	Since $F_\varphi(r,s_1:s_\alpha)$ is a maximum typical 2-inducing   multifan, 
	$\varphi(ru)\in \{\alpha+1, \Delta\}$. 
	Assume instead that $\varphi(ru)=\Delta$. 
	If $\alpha=1$, then we are done by exchanging the 
	roles of $2$ and $\Delta$. Thus we assume that $\alpha\ge 2$. 
	Shift from $s_2$ to $s_{\alpha-1}$, 
	color $rs_1$ by 2 and uncolor $rs_\alpha$. 
	Then $F^*=(r,rs_\alpha,s_\alpha, rs_{\alpha-1},s_{\alpha-1}, \ldots, rs_1, s_1)$ is an $\alpha$-inducing multifan such that $\Delta$
	is the last $\alpha$-inducing color. 
	Now,  relabeling colors and vertices 
	in $F^*$ by making $F^*$ typical 2-inducing  yields the desired assumption.
	\qed

	\begin{CLA}\label{u-sharecolor}
		For any $z\in N_{\Delta-1}(u)\setminus V(F)$ and any $(F,\varphi)$-stable $\varphi'\in \CC^\Delta(G-rs_1)$,  if $\varphi'(ru)=\alpha+1$ and $\pbar'(z)=\alpha+1$, 
		then $\varphi'(uz)\in \pbar'(F)\setminus\{1\}$. 
	\end{CLA}
	\proof[Proof of Claim~\ref{u-sharecolor}]
	Assume to the contrary that $\varphi'(uz)\in \{1, \alpha+2,\dots, \Delta-1\}$. 
	We first claim that $\varphi'(uz)\ne 1$. 
	As otherwise,  $P_r(1,\alpha+1, \varphi')=ruz$, 
	contradicting Lemma~\ref{thm:vizing-fan1} \eqref{thm:vizing-fan1b} that $r$
	and $s_\alpha$ are $(1,\alpha+1)$-linked with respect to $\varphi'$. Let 
	$\varphi'(uz)=\tau\in  [\alpha+2,\Delta-1]$, and $w_1\in N_{\Delta-1}(r)$ such that $\varphi'(rw_1)=\tau$.   By Lemma~\ref{Lemma:extended multifan} \eqref{Evizingfan-a},  $uz\in P_r(1,\tau,\varphi'')$
	for every  $(L,\varphi')$-stable $\varphi'' \in \CC^{\Delta}(G-rs_1)$, 
	where $L=(F,ru,u,uz,z)$ is a lollipop. 
	Applying Lemma~\ref{Lemma:pseudo-fan0} \eqref{fan-and-x} on $L$ with $u$ playing the role of $w$,  
	we find a sequence of distinct vertices $w_1, \ldots, w_t\in  \{s_{\alpha+1}, \ldots, s_{\Delta-2} \}$ that forms either a stable rotation or a near stable rotation. 
	
	Assume first that $w_1, \ldots, w_t$ is a stable rotation,
	which in particular gives $P_r(1,\tau,\varphi')=P_{w_t}(1,\tau, \varphi')$.  
	 By Lemma~\ref{Lemma:extended multifan} \eqref{Evizingfan-a}, $uz\in P_r(1,\tau,\varphi')$. 
	If  $P_{r}(1,\tau,\varphi')$ meets $z$ before $u$, or 
	equivalently, $P_{w_t}(1,\tau,\varphi')$  meets $u$ before $z$, 
	we do the following operations:
	\[
	\begin{bmatrix}
	P_{[r,z]}(1,\tau,\varphi') & ru  &uz\\
	1/\tau & \alpha+1 \rightarrow \tau & \tau \rightarrow \alpha+1
	\end{bmatrix}.
	\]
	Denote the new coloring by $\varphi''$. Now $(r,rs_1,s_1,\ldots,s_{\alpha})$ is a multifan, but $\pbar''(s_{\alpha})=\pbar''(r)=\alpha+1$, giving a contradiction to Lemma~\ref{thm:vizing-fan1} \eqref{thm:vizing-fan1a}.
	Thus $P_{r}(1,\tau,\varphi')$ meets $u$ before $z$, or  equivalently, $P_{w_t}(1,\tau,\varphi')$  meets $z$ before $u$. 
	Shift  from $w_1$
	to $w_t$ to get $\varphi''$. Then $P_{r}(1,\tau, \varphi'')$ meets $z$ before $u$, giving back to the previous case as $\varphi''$ is $(F,\varphi')$-stable.

	Assume now that $w_1, \ldots, w_t$ is a near stable rotation, i.e., $\pbar'(w_t)=\alpha+1$.  
If $z\ne w_t$, then we 
	shift  from $w_1$ to $w_t$, 
	and do $ru: \alpha+1\rightarrow \tau$, $uz: \tau\rightarrow \alpha+1$.
	Denote the new coloring by $\varphi''$. As $\varphi''$ is $(F,\varphi')$-stable and so is $(F,\varphi)$-stable,
	we see that  $F^*=(F, rw_t, w_t, rw_{t-1}, w_{t-1}, \ldots, rw_1, w_1)$
	is a multifan that contains more vertices than
	$F$ does,  
	showing a 
	contradiction to the choice of $F$.  
	
Thus we assume that $z=w_t$. Since $\varphi'(rz) \ne \varphi'(uz)=\tau$, we have $t\ge 2$. 
	Note that  $uz\in P_r(1,\tau,\varphi')=P_{w}(1,\tau,\varphi')$ for some vertex $w\in V(G)\setminus(V(F)\cup \{w_1, \ldots, w_t\})$.
	 If  $P_{w}(1,\tau,\varphi')$ meets $u$ before $z$, 
	we do the following operations:
	\[
	\begin{bmatrix}
	P_{[w,u]}(1,\tau,\varphi') & ru & uz\\
	1/\tau & \alpha+1\rightarrow 1& \tau \rightarrow \alpha+1
	\end{bmatrix}.
	\]
	Denote the new coloring by $\varphi''$. Now $(r,rs_1,s_1,\ldots,s_{\alpha})$ is a multifan, but $\pbar''(s_{\alpha})=\pbar''(r)=\alpha+1$, giving a contradiction to Lemma~\ref{thm:vizing-fan1} \eqref{thm:vizing-fan1a}.
	 If  $P_{w}(1,\tau,\varphi')$ meets $z$ before $u$,
	we do the following operations:
	\[
	\begin{bmatrix}
	P_{[w,z]}(1,\tau,\varphi') &  w_1:w_t & rw_t=rz & ru &uz\\
	1/\tau &  \text{shift} & \varphi'(rz)\rightarrow 1 & \alpha+1 \rightarrow \tau & \tau \rightarrow \alpha+1
	\end{bmatrix}.
	\]
	Denote the new coloring by $\varphi''$. Now $(r,rs_1,s_1,\ldots,s_{\alpha})$ is a multifan, but $\pbar''(s_{\alpha})=\pbar''(r)=\alpha+1$, giving a contradiction to Lemma~\ref{thm:vizing-fan1} \eqref{thm:vizing-fan1a}.
	\qed

	By Claim~\ref{u-sharecolor},    $\tau \in \{2, \ldots, \alpha+1, \Delta\}$. 
	Applying Lemmas~\ref{Lem:2-non-adj1} and \ref{Lem:2-non-adj2}, we have the following claim.

	\begin{CLA}\label{u-nonadj}
		Let  $z\in N_{\Delta-1}(u)\setminus V(F)$  and any $(F,\varphi)$-stable $\varphi'\in \CC^\Delta(G-rs_1)$ such that $\varphi'(ru)=\alpha+1$ and $\pbar'(z)=\alpha+1$, and let  $\varphi'(uz)=\tau$.
		Then 	$\tau\in \pbar'(F)\setminus \{1\}$, and 
		$u\not\sim s_1,s_\alpha$ 	if $\tau =\Delta$; and $u\not\sim s_{\tau-1}, s_\tau$ if $\tau \in [2,\alpha+1]$.
	\end{CLA}

	\begin{CLA}\label{pseudo-fan}
		Suppose that $N_{\Delta-1}(r)=N_{\Delta-1}(u)$ for every $u\in N_\Delta(r)$.
		Then for every  $(F,\varphi)$-stable coloring $\varphi'\in \CC^\Delta(G-rs_1)$,  $N_{\Delta-1}[r]$ is $\varphi'$-elementary. In particular, $N_{\Delta-1}[r]$ is the vertex set of 
		 a typical 2-inducing pseudo-multifan  
		with respect to $rs^*$ and  $\varphi^*\in \CC^\Delta(G-rs^*)$  for some $s^*\in N_{\Delta-1}(r)$. 
	\end{CLA}
	
	\proof[Proof of Claim~\ref{pseudo-fan}] 
	Assume to the contrary that there exists an $(F,\varphi)$-stable coloring $\varphi'\in \CC^\Delta(G-rs_1)$ such that $N_{\Delta-1}[r]$ is not  $\varphi'$-elementary.  
	Since $V(F)$ is $\varphi'$-elementary, there exists $z\in N_{\Delta-1}[r]\setminus V(F)$
	such that $\pbar'(z)\in \pbar'(F)$ or there exists $z^*\ne z$ with $z^*\in N_{\Delta-1}[r]\setminus V(F)$
	such that $\pbar'(z)=\pbar'(z^*)$.  Let $\pbar'(z)=\delta$. 
	If $\delta\in \pbar'(F)$, then $z$ and $r$
	are $(1,\delta)$-unlinked,   so we do $(\delta,1)-(1,\alpha+1)$-swaps at $z$; 
	if  $\pbar'(z)=\pbar'(z^*)$, we may assume, without loss of generality, that $z$ and $r$
	are $(1,\delta)$-unlinked, we again do  $(\delta,1)-(1,\alpha+1)$-swaps at $z$. 
	In either case, we find an $(F,\varphi')$-stable coloring $\varphi''\in \CC^\Delta(G-rs_1)$ such that 
	$\pbar''(z)=\alpha+1$.  Since for any $u\in N_\Delta(r)$,  it holds that  $N_{\Delta-1}(r)=N_{\Delta-1}(u)$,
	we can choose $u\in N_\Delta(r)$  such that $\varphi''(ur)=\alpha+1$, where $\alpha+1$ is the last 2-inducing color of $F_{\varphi''}(r,s_1:s_\alpha)$.  Since $N_{\Delta-1}(r)=N_{\Delta-1}(u)$, we have $uz\in E(G)$ and so 
	$L=(F_{\varphi''}(r,s_1:s_\alpha), ru, u, uz, z)$ is a lollipop with respect to $\varphi''$. 
	By Claim~\ref{u-nonadj},  $u$ is not adjacent to at least one vertex in $N_{\Delta-1}(r)$, 
	which in turn shows $N_{\Delta-1}(r)\ne N_{\Delta-1}(u)$, giving a contradiction. 
	
	Therefore, for every $(F,\varphi)$-stable coloring $\varphi'\in \CC^\Delta(G-rs_1)$, it holds that  $N_{\Delta-1}[r]$ is $\varphi'$-elementary.  Consequently, there 
	is  a  pseudo-multifan  with vertex set $N_{\Delta-1}[r]$. 
	By renaming  colors and  vertices from 
	$N_{\Delta-1}(r)$, we can assume the  pseudo-multifan with vertex set $N_{\Delta-1}[r]$ is typical. By Lemma~\ref{2-inducing}, we can further assume that the pseudo-multifan is typical 2-inducing. 
	\qed 
	
	{\bf By Claim~\ref{pseudo-fan}, it suffices to only show Theorem~\ref{Thm:vizing-fan2b} \eqref{common2}}.
	Assume to the contrary that there exists  $u\in N_{\Delta}(r)$ such that 
	$N_{\Delta-1}(u)\setminus N_{\Delta-1}(r)\ne \emptyset$. 
	
	\begin{CLA}\label{x-missing-color}
		For every $z\in N_{\Delta-1}(u)\setminus N_{\Delta-1}(r)$, there is  an $(F,\varphi)$-stable coloring 
		$\varphi'\in \CC^\Delta(G-rs_1)$ such that
		$\varphi'(ru)=\alpha+1$ and  $\pbar'(z)=\alpha+1$. 
	\end{CLA}
	\proof[Proof of Claim~\ref{x-missing-color}]
	%
	By Claim~\ref{ux-color}, we assume  $\varphi(ru)=\alpha+1$. 
	Let $\pbar(z)=\delta$.  If $\delta=\alpha+1$, we simply let $\varphi'=\varphi$.  So    $\delta\ne \alpha+1$. 
	If $\delta\in \pbar(F)$, 
	we let  $\varphi'$ be obtained from $\varphi$ by doing $(\delta,1)-(1,\alpha+1)$-swaps
	at $z$.   This gives that $\pbar'(z)=\alpha+1$. 
By Lemma~\ref{thm:vizing-fan1} \eqref{thm:vizing-fan1b}, $\varphi'$
	is $(F,\varphi)$-stable and  $\varphi'(ru)=\varphi(ru)=\alpha+1$. Thus $\varphi'$  is a desired coloring. 
	
	Assume now that $\delta\in [\alpha+2,\Delta-1]$. 
	If there is an $(F,\varphi)$-stable $\varphi'' \in \CC^{\Delta}(G-rs_1)$ with $\varphi''(ru)=\alpha+1$
	such that $z\not\in P_r(1,\delta, \varphi'')$ (so $z$ and $r$ are $(1,\delta)$-unlinked), 
	let $\varphi'$ be obtained from $\varphi''$
	by doing $(\delta,1)-(1,\alpha+1)$-swaps
	at $z$. Since 
	$\varphi''(ru)=\alpha+1$ and $r$ and $s_\alpha$
	are $(1,\alpha+1)$-linked with respect to $\varphi''$
	by Lemma~\ref{thm:vizing-fan1} \eqref{thm:vizing-fan1b},  
	it holds that  $\varphi'$ is $(F,\varphi'')$-stable and so $(F,\varphi)$-stable  with
	$\varphi'(ru)=\varphi''(ru)=\alpha+1$. Thus, $\varphi'$ is a desired coloring and 
	we are done.  
	Therefore every  $(F,\varphi)$-stable $\varphi'' \in \CC^{\Delta}(G-rs_1)$ with $\varphi''(ru)=\alpha+1$ satisfies  $z\in P_r(1,\delta, \varphi'')$. 
	Applying Lemma~\ref{Lemma:pseudo-fan0} \eqref{only-fan} with  $z$ playing the role of $w$,  there exists 
	$w_t\in N_{\Delta-1}(r)\setminus V(F)$ such that $\pbar(w_t)=\delta$
	and $w_t$ and $r$ are $(1,\delta)$-linked with respect to $\varphi$. 
	This is a contradiction  by noting $w_t\ne z$,  since $\varphi$ is $(F,\varphi)$-stable but   $z\notin P_r(1,\delta, \varphi)$. 
	\qed 
	
	\begin{CLA}\label{u-outneighbor-x}
		 $|N_{\Delta-1}(u)\setminus N_{\Delta-1}(r) |\ge 2$. 
	\end{CLA}
	\proof[Proof of Claim~\ref{u-outneighbor-x}]
	Let $x\in N_{\Delta-1}(u)\setminus N_{\Delta-1}(r)$. 
	By Claim~\ref{x-missing-color},  we choose an $(F,\varphi)$-stable coloring from 
	$\CC^\Delta(G-rs_1)$ and call it still $\varphi$ such that $\varphi(ru)=\alpha+1$ and $\pbar(x)=\alpha+1$. 
	By Claim~\ref{u-nonadj},  $\varphi(ux)\in  \{2, \ldots, \alpha+1, \Delta\}$. 
	If $|V(F)|\ge 3$, then 
	Claim~\ref{u-nonadj} gives that 
	$|N_{\Delta-1}(u)\setminus N_{\Delta-1}(r) |\ge 2$. 
	Thus we have 	$V(F)=\{r,s_1\}$. Consequently, $\alpha+1=2$, and $\varphi(ux)=\Delta$ by 
	the fact that $\varphi(ux)\in  \{2,  \Delta\}$. 
	We may assume further that 
	$
	N_{\Delta-1}(u)\setminus N_{\Delta-1}(r) =\{x\}.
	$
	
	By Claim~\ref{u-nonadj}, $u\not\sim s_1$. 
	We consider two cases.
	Assume first that there exists an $(F,\varphi)$-stable $\varphi'\in \CC^\Delta(G-rs_1)$ 
	such that $N_{\Delta-1}[r]$ is not $\varphi'$-elementary.  
	By exchanging the roles of $2$ and $\Delta$ if necessary, we may assume 
	$\varphi'(ru)=2$.
	Since $V(F)$ is $\varphi'$-elementary, there exists $z\in N_{\Delta-1}(r)\setminus V(F)$
	such that $\pbar'(z)\in \pbar'(F)$ or there exists $z^*\ne z$ with $z^*\in N_{\Delta-1}(r)\setminus V(F)$
	such that $\pbar'(z)=\pbar'(z^*)$. Let $\pbar'(z)=\delta$. 
	If $\delta\in \pbar'(F)$,   then as $r$ and $z$ are $(1,\delta)$-unlinked, we do $(\delta,1)-(1,2)$-swaps at $z$; 
	if  $\pbar'(z)=\pbar'(z^*)$, we may assume, without loss of generality, that $z$ and $r$
	are $(1,\delta)$-unlinked, we again do  $(\delta,1)-(1,2)$-swaps at $z$. 
	In either case, we find an $(F,\varphi' )$-stable coloring $\varphi''\in \CC^\Delta(G-rs_1)$ with  
	$\varphi''(ru)=\varphi'(ru)=2$
	and $\pbar''(z)=2$. Note that $z\in N_{\Delta-1}(u)$ since  $u\not\sim s_1$, $s_1\ne z$,  and $N_{\Delta-1}(u)\setminus  N_{\Delta-1}(r)=\{x\}$. 
	By	Claim~\ref{u-sharecolor}, 
	$\varphi''(uz)\in \{2, \Delta\}$, which implies $\varphi''(uz)=\Delta$ by noting $\varphi''(ru)=2$.
	Furthermore, we assume $uz\in P_{s_1}(1,\Delta,\varphi'')=P_r(1,\Delta,\varphi'')$ (otherwise, after a $(1,\Delta)$-swap on the chain containing $uz$, we obtain a contradiction to Claim~\ref{u-sharecolor}). 
	Since $\varphi''(ru)=2$ and $\varphi''(uz)=\Delta$, $\varphi''(ux)\ne 2, \Delta$. 
	Thus $\varphi''(ux)\in \{1,3, 4, \ldots, \Delta-1\}$, which 
	implies $\pbar''(x) \ne 2$ by 
	Claim~\ref{u-sharecolor}.  
	Let $\pbar''(x)=\tau$ and $\varphi''(ux)=\lambda$. 
	Note that if $\tau=\Delta$ then $\lambda\ne 1$,  as $uz\in P_{s_1}(1,\Delta,\varphi'')=P_r(1,\Delta,\varphi'')$. 
	Thus if $\tau=\Delta$ or $1$, we do  $(\tau,1)-(1,2)$-swaps at $x$. 
	As the color of $ux$ is not $\Delta$ after these swaps, we get  a contradiction to Claim~\ref{u-sharecolor}. 
	Thus, we assume that $\tau\in [3,\Delta-1]$, and that $P_x(1,\tau,  \varphi''')=P_r(1,\tau,  \varphi''')$
	for any $(L,\varphi'')$-stable  coloring $\varphi'''$,  where $L=(F,ru, u, uz,z)$ is a lollipop.  Let  $w_1\in N_{\Delta-1}(r)$ such that $\varphi''(rw_1)=\tau$. 
	Applying  Lemma~\ref{Lemma:pseudo-fan0} \eqref{fan-and-x} on $L$ with $x$ playing the role of $w$, 
	we find a sequence of distinct vertices $w_1, \ldots, w_t\in  \{s_{\alpha+1}, \ldots, s_{\Delta-2} \}$ that forms either a stable rotation or a near stable rotation. 
	As $x$ and $r$ are $(1,\tau)$-linked, we conclude that $w_1, \ldots, w_t$ form a near stable rotation and so $\pbar''(w_t)=2$. 
	As $\varphi''(uz)=\Delta$, $\varphi''(ur)=2$, if $w_t\ne z$, then  
	$\varphi''(uw_t)\in \{1, 3, 4, \ldots, \Delta-1 \}$. 
	This gives a contradiction to Claim~\ref{u-sharecolor}. 
	Thus we assume that $w_t=z$. 
	Notice that  $r\in P_{s_1}(2,\tau,\varphi'')$ by the maximality of $|V(F)|$.
	Since 	$r\in P_x(2,\tau,\varphi'')$ by Claim~\ref{u-sharecolor}, we have 
	$r\in P_x(2,\tau,\varphi'')=P_{s_1}(2,\tau,\varphi'')$. 
	So $w_t$ is $(2,\tau)$-unlinked with any of $s_1, x$ and $r$ with respect to $\varphi''$. 
	We do a $(2,\tau)$-swap at $w_t$ and then shift  from $w_1$
	to $w_t$. This gives a coloring such   that $s_1$ and $x$ are $(2,\tau)$-unlinked with respect to the coloring. 
	Again, with respect to the current coloring,  $r$ and $s_1$ are $(2,\tau)$-linked by the maximality of $|V(F)|$.
	We do a $(2,\tau)$-swap at $x$ to get a coloring $\varphi'''$. 
	Note that $\varphi'''(ru)=\varphi''(ru)=2$, $\varphi'''(ux)=\varphi''(ux)=\lambda$,
	$\pbar'''(x)=2$, and $\varphi'''(uz)=\varphi''(uz)=\Delta$.
	Therefore, $\varphi'''(ux)=\lambda\in \{1,3,4,\ldots, \Delta-1\}$, 
	showing  a contradiction to Claim~\ref{u-sharecolor}.

	Thus  we assume that   $N_{\Delta-1}[r]$  is $\varphi'$-elementary for every $(F,\varphi)$-stable $\varphi'\in \CC^\Delta(G-rs_1)$. 
	In particular, $N_{\Delta-1}[r]$  is $\varphi$-elementary, and  as $|V(F)|=2$ and $F$  is maximum at $r$, 
	we know that $N_{\Delta-1}[r]$ 
	is contained in a pseudo-multifan $S=S_{\varphi}(r,s_1:s_1:s_{\Delta-2})$. Let $\delta\in \pbar(S)\setminus \pbar(F)$.
	By Lemma~\ref{pseudo-fan-ele}~\eqref{pseudo-b}, $\pbar^{-1}_S(\delta)$
	is $(2,\delta)$- and $(\delta,\Delta)$-linked with $s_1$ 
	and the corresponding chains contain the vertex $r$ with respect to $\varphi$.  
	Recall that $\varphi(ru)=2, \varphi(ux)=\Delta$, and $\pbar(x)=2$.
	Let $\varphi'$ be obtained from $\varphi$ by doing 
	a
	$(2,\delta)-(\delta, \Delta)-(\Delta,1)-(1,2)$-swap at $x$. 
	Since $\varphi'$ is $(F,\varphi)$-stable, $\varphi'(ru)=2$, $\varphi'(ux)=\delta$, and $\pbar'(x)=2$, we get 
	a contradiction to Claim~\ref{u-sharecolor}. 
	\qed 
	
	\begin{CLA}\label{u-2-out-neighbor}
		Let $x,y\in N_{\Delta-1}(u)\setminus N_{\Delta-1}(r)$ be distinct, and $\varphi'\in \CC^\Delta(G-rs_1)$
		be any $(F,\varphi)$-stable coloring with $\varphi'(ru)=\alpha+1$.  
		Suppose $\pbar'(x)\in \pbar'(F)$ and $\pbar'(x)\ne 1$.  Then $\pbar'(y)\not\in \pbar'(F)$. Furthermore, $y$ and $r$
		are $(1,\pbar'(y))$-linked with respect to $\varphi'$. 
	\end{CLA}
	\proof[Proof of Claim~\ref{u-2-out-neighbor}]
	
	The second  part of the claim  follows easily from the first part. Since otherwise, a $(1,\pbar'(y))$-swap at $y$
	implies that $1$ is missing at $y$, contradicting the first part. 

	Assume to the contrary that $\pbar'(x)\in \pbar'(F)$ and $\pbar'(y)\in \pbar'(F)$.  
	We claim that we may assume $\pbar'(x)=\pbar'(y)=\alpha+1$ or 
	$\pbar'(x)=\alpha+1$ and $\pbar'(y)=1$.  
	By doing  $(\pbar'(x),1)-(1,\alpha+1)$-swaps at $x$, 
	we  assume that $\pbar'(x)=\alpha+1$. 
	Since $1,\alpha+1\in \pbar'(F)$, we still 
	have $\pbar'(y)\in \pbar'(F)$. 
	If $\pbar'(y)=\alpha+1$, then we are done. 
	Otherwise, doing a $(1,\pbar'(y))$-swap at $y$ gives a desired coloring.
Let $\varphi'(ux)=\tau$ and $\varphi'(uy)=\lambda$.
	We consider now two cases to finish the proof of Claim~\ref{u-2-out-neighbor}. 
	
\smallskip 
	{\noindent \bf \setword{Case A}{Case A}: $\pbar'(x)=\pbar'(y)=\alpha+1$.}
\smallskip

	By Claim~\ref{u-sharecolor}, $\tau, \lambda\in \pbar'(F)\setminus\{1\}$. 
	Assume, without loss of generality, that $\tau\ne \Delta$. 
	Then $\tau\in \{2,\ldots, \alpha+1\}$ is a 2-inducing color of $F$, since $F$
	is assumed to be typical 2-inducing.  
	By Lemma~\ref{Lemma:extended multifan}~\eqref{Evizingfan-d} 
	that $r\in P_{s_\alpha}(\alpha+1,\Delta)=P_{s_1}(\alpha+1,\Delta)$, 
	we know $\lambda\ne \Delta$. 
	Thus $\lambda\in \{2,\ldots, \alpha+1\}$ is also a 2-inducing color. 
	By symmetry between $x$ and $y$, we assume  $\lambda \prec \tau$. 
	Shift from $s_2$ to $s_{\lambda-1}$, uncolor $rs_\lambda$, 
	then color $rs_1$ by 2.  Denote the resulting coloring by $\varphi''$. 
	Now  $F^*=(r, rs_\lambda, s_\lambda, rs_{\lambda+1}, s_{\lambda+1}, \ldots, rs_\alpha, s_\alpha, rs_{\lambda-1}, s_{\lambda-1}, \ldots, rs_1,s_1)$ is a new  multifan with respect to $\varphi''$ that has the same vertex set 
	as $F_{\varphi'}(r,s_1:s_\alpha)$. In this new multifan $F^*$, 
	$\lambda$ is itself a $\lambda$-inducing color,  $\tau$ is a $(\lambda+1)$-inducing color, and $\alpha+1$ is the last $(\lambda+1)$-inducing color. We can further assume that $F^*$ is typical by relabeling colors and vertices. 
	However, $r\in P_y(\alpha+1, \lambda,\varphi'')$, shows a contradiction to 
	Lemma~\ref{Lemma:extended multifan}~\eqref{Evizingfan-d}  that $r\in P_{s_\alpha}(\alpha+1,\lambda, \varphi'')=P_{s_{\lambda}}(\alpha+1,\lambda,\varphi'')$.  
	
\smallskip 
	{\bf \noindent Case B: $\pbar'(x)=\alpha+1$ and $\pbar'(y)=1$.}
	\smallskip

	We  assume that $x$ and $y$ are $(1,\alpha+1)$-linked with respect to $\varphi'$. For otherwise, a $(1,\alpha+1)$-swap at 
	$y$ reduces the problem to \ref{Case A}.

	We show that $\tau, \lambda \ne \Delta$.  If this is not the case, then  by swapping colors along $P_{[x,y]}(1,\alpha+1,\varphi')$ and exchanging the roles of $x$ and $y$ if necessary,  we assume that 
	$\tau \ne \Delta$ and $\lambda= \Delta$. Let $\varphi''$ be obtained 
	from $\varphi'$ by a $(1,\Delta)$-swap at $y$. 
	By Lemma~\ref{Lemma:extended multifan}~\eqref{Evizingfan-d}, $r\in P_{s_1}(\alpha+1,\Delta, \varphi'')=P_{s_\alpha}(\alpha+1,\Delta,\varphi'')$. 
	Thus, we can do an $(\alpha+1,\Delta)$-swap at $y$ without affecting the coloring of $F_{\varphi''}(r,s_1:s_\alpha)$ and $\varphi''(ru)$. 
	Thus, let $\varphi^*=\varphi''/P_y(\alpha+1,\Delta, \varphi'')$. 
	We see that $P_r(1,\alpha+1,\varphi^*)=ruy$, showing a contradiction to Lemma~\ref{thm:vizing-fan1}~\eqref{thm:vizing-fan1b} that $r$
	and $s_\alpha$ are $(1,\alpha+1)$-linked with respect to $\varphi^*$. 
	
	Since  $\tau, \lambda \ne \Delta$, 
	both $\tau$ and $\lambda$ are 2-inducing colors of $F$ 
	by Claim~\ref{u-sharecolor}.  
	By swapping colors along $P_{[x,y]}(1,\alpha+1,\varphi')$ and exchanging the roles of $x$
	and $y$ if necessary, we assume 
	$  \tau \prec \lambda $.  
	Note that $r\in P_{s_1}(\lambda,\Delta,\varphi')=P_{s_{\lambda-1}}(\lambda,\Delta,\varphi')$  and 
	$r\in P_{s_1}(\alpha+1,\Delta,\varphi')=P_{s_\alpha}(\alpha+1,\Delta,\varphi')$ by Lemma~\ref{Lemma:extended multifan}~\eqref{Evizingfan-c}  and~\eqref{Evizingfan-d}, respectively. 
	Let $\varphi''$ be obtained from $\varphi'$ by doing  a 
	$ (1,\Delta)-(\Delta, \lambda )-(\lambda,1)-(1,\Delta)-(\Delta,\alpha+1)$-swap at $y$.
	Note that $\varphi''$ is $(F,\varphi')$-stable, and 
	that $P_r(1,\alpha+1,\varphi'')=ruy$, showing a contradiction to Lemma~\ref{thm:vizing-fan1}~\eqref{thm:vizing-fan1b} that $r$
	and $s_\alpha$ are $(1,\alpha+1)$-linked with respect to $\varphi''$. 
	\qed 
	
	By  Claim~\ref{x-missing-color} and Claim~\ref{u-outneighbor-x},  we let 
	$x,y\in N_{\Delta-1}(u)\setminus N_{\Delta-1}(r)$  with $x\ne y$, 
	and assume that 
	$\varphi(ru)=\alpha+1$  and $\pbar(x)=\alpha+1$. 
	By Claim~\ref{u-2-out-neighbor}, 
	we also assume that $\pbar(y)=\delta \in [\alpha+2,\Delta-1]$
	and $y$ and $r$ are $(1,\delta)$-linked with respect to such a coloring $\varphi$. 
	Let $w_1\in N_{\Delta-1}(r)$ such that 
	$\varphi(rw_1)=\delta$  and $L=(F,ru,u, ux,x)$. 
	By Claim~\ref{u-2-out-neighbor}, 
	for any $L$-stable $\varphi' \in \CC^{\Delta}(G-rs_1)$,   it holds that 
	$y\in P_r(1,\delta,\varphi')$. 
	Applying Lemma~\ref{Lemma:pseudo-fan0}~\eqref{fan-and-x} 
	on $L$ with $y$ playing the role of $w$,
we find a sequence of distinct vertices $w_1, \ldots, w_t\in  \{s_{\alpha+1}, \ldots, s_{\Delta-2} \}$ that forms either a stable rotation or a near stable rotation. 
Since $y$ and $r$ are $(1,\delta)$-linked with respect to $\varphi$, $w_1, \ldots, w_t$
is a near stable rotation, i.e.,  $\pbar(w_t)=\alpha+1$.

	\begin{CLA}\label{u-3-neighbor}
		$|N_{\Delta-1}(u)\setminus N_{\Delta-1}(r) |\ge 3$. 
	\end{CLA}
	\proof[Proof of Claim~\ref{u-3-neighbor}]
		Let $\varphi(ux)=\tau$ and $\varphi(uy)=\lambda$.  Since 
	$\varphi(ru)=\alpha+1$, we have $\alpha+1\notin \{\tau,\lambda,\delta\}$. 
	By Claim~\ref{u-nonadj}, $\tau\in \pbar(F)\setminus\{1\}$, and 
	\begin{numcases}{}
	s_{1}, s_{\alpha} \not\in N_{\Delta-1}(u) & \text{if $\tau=\Delta$}, \label{2-small-neighborsa}\\
	s_{\tau-1}, s_{\tau} \not\in N_{\Delta-1}(u) & \text{if $\tau\ne \Delta$}. \label{2-small-neighborsb}
	\end{numcases}
	
	We  then show that 
	\begin{numcases}{}
	s_{1}, s_{\alpha} \not\in N_{\Delta-1}(u) & \text{if $\lambda=\Delta$}, \label{2-small-neighbors2a}\\
	s_{\lambda-1}, s_{\lambda} \not\in N_{\Delta-1}(u) & \text{if $\lambda\ne \Delta$}. \label{2-small-neighbors2b}
	\end{numcases}

	To see this, let $\varphi'$ be obtained from $\varphi$
	by first doing a $(1,\alpha+1)$-swap at both $x$ and $w_t$, 
	and then  shift from $w_1$ to $w_t$. 
	Now, $\pbar'(r)=\delta$ and $\varphi'(ux)=\varphi(ux)=\tau$. 
	Let $\varphi''=\varphi'/P_y(\alpha+1,\delta,  \varphi')$. 
	Note that $\varphi''(ux)=\varphi'(ux)=\tau$ and $\varphi''(uy)=\varphi'(uy)=\lambda$.  Applying Claim~\ref{u-sharecolor} to the coloring $\varphi''$, we get  
	$\varphi''(uy)=\lambda\in \pbar''(F)\setminus\{\delta\}$.  
	As $\tau,\lambda,\delta,\alpha+1 \in \pbar''(F)$ and they are all distinct, 
	$|V(F)|\ge |\{\delta,\tau,\lambda,\alpha+1\}|-1 = 3$.  Then~\eqref{2-small-neighbors2a} and~\eqref{2-small-neighbors2b} follow   from  Claim~\ref{u-nonadj}. 
	These two facts, 
	together with~\eqref{2-small-neighborsa} and~\eqref{2-small-neighborsb},
	imply 
	$
	\text{either} s_1,s_\alpha, s_{\lambda-1}, s_{\lambda} \not\in N_{\Delta-1}(u),  \text{or}  s_1,s_\alpha, s_{\tau-1}, s_{\tau} \not\in N_{\Delta-1}(u).
	$
	Note that $s_1\ne s_\alpha$ by $|V(F)|\ge 3$. 
	We obtain $
	|N_{\Delta-1}(u)\setminus N_{\Delta-1}(r)|\ge  3 
	$	from the above  unless 
either ${\lambda}={\alpha}=2$ or ${\tau}={\alpha}=2$.

	Therefore we assume  $\alpha=2$ and $\{\lambda,\tau\}=\{2,\Delta\}$.  Furthermore, we may assume that  $|N_{\Delta-1}(u)\setminus N_{\Delta-1}(r)|=2$, 
	 since~\eqref{2-small-neighborsa} to~\eqref{2-small-neighbors2b}
	imply that  $s_1, s_2\not\in N_{\Delta-1}(u)$. Therefore 
	$N_{\Delta-1}(r)\setminus\{s_1,s_2\}\subseteq N_{\Delta-1}(u)$. 
	In particular, $w_t\in N_{\Delta-1}(u)$. Since $r$
	and $s_\alpha$ are $(1,\alpha+1)$-linked with respect to $\varphi$
	and $\pbar(w_t)=\alpha+1$, it follows that $\varphi(uw_t)\ne 1$. This, together with the facts that  $\pbar(F)=\{1,2,3,\Delta\}$, $\varphi(ru)=3$, and  $\{\lambda,\tau\}=\{2,\Delta\}$,  implies  that $\varphi(uw_t)\in [4,\Delta-1]$, 
	showing a contradiction to Claim~\ref{u-sharecolor}. 
	\qed 
	
	By Claim~\ref{u-3-neighbor}, 
	let $z\in N_{\Delta-1}(u)\setminus N_{\Delta-1}(r)$ with $z\ne x,y$. 
	By Claim~\ref{u-2-out-neighbor},  we assume   $\pbar(z)=\lambda \in [\alpha+2,\Delta-1]$,
	  and $z$ and $r$ are $(1,\lambda)$-linked with respect to $\varphi$.  Since  $\pbar(y)=\delta$, and also $y$ and $r$ are $(1,\delta)$-linked with respect to $\varphi$, we have  $\lambda\ne \delta$.

	Recall $w_1, \ldots, w_t$ is a near stable rotation at $r$ with $\varphi(rw_1)=\delta=:\delta_1$. Let $\pbar(w_i)=\delta_{i+1}$ for each $i\in [1,t-1]$.  
	As $z$ and $r$ are $(1,\lambda)$-linked with respect to $\varphi$ and $w_i$ and $r$ are $(1,\delta_{i+1})$-linked for each $i\in [1,t-1]$, $\lambda \ne \delta_i$ for each $i\in [2,t]$. 
	Let $\lambda_1=\lambda$ and $w_1^*$ be the neighbor of $r$
	such that $\varphi(rw_1^*)=\lambda_1$. 
	For any $(L,\varphi)$-stable coloring 
	$\varphi'\in \CC^\Delta(G-rs_1)$, $z\in P_r(1,\lambda,\varphi')$. 
	Applying Lemma~\ref{Lemma:pseudo-fan0}~\eqref{fan-and-x} on $L=(F,ru, u,ux,x)$ and $z$,
	we find a sequence of distinct vertices $w^*_1, \ldots, w^*_{k}\in  \{s_{\alpha+1}, \ldots, s_{\Delta-2} \}$ that forms either a stable rotation or a near stable rotation.
	If $\pbar(w_k^*)=\lambda_1$, then since $w_k^*$ and $r$ are $(1,\lambda_1)$-linked, 
	a $(1,\lambda_1)$-swap at $z$ gives a contradiction to Claim~\ref{u-2-out-neighbor}. 
	Thus  $\pbar(w_k^*)=\alpha+1$. Let  $\pbar(w^*_i)=\lambda_{i+1}$ for each $i\in [1,k-1]$.
	
	Recall that  $w_1^*\ne w_i$ for each $i\in[1,t]$. 
	Furthermore, as  $w_i$ and $r$ are $(1,\delta_{i+1})$-linked for each $i\in [1,t-1]$ and  
	$w_j^*$ and $r$ are $(1,\lambda_{j+1})$-linked for each $j\in [1,k-1]$, $w_1^*\ne w_i$
	for each $i\in [1,t]$ implies that $\lambda_2 \not\in\{\delta_1,\ldots, \delta_t\}$.
	Consequently, $w^*_2\ne w_i$ for each $i\in[1,t]$. Repeating the same process,
	we get  $w_j^*\ne w_i$ for each $j\in [1,k]$ and each $i\in[1,t]$.

	We claim that $w_t$ and $x$ are $(1,\alpha+1)$-linked with respect to $\varphi$. 
	For otherwise, first doing a $(1,\alpha+1)$-swap at $w_t$, then shift from $w_1$
	to $w_t$ gives a coloring $\varphi'$  such that $\varphi'(ru)=\varphi(ru)=\alpha+1$, $\pbar'(y)=\pbar'(r)=\delta_1$, while $\pbar'(x)=\alpha+1$.
	Based on $\varphi'$, after doing a $(1,\delta_1)$-swap on all $(1,\delta_1)$-chains in $G-rs_1$, we obtain an $(F,\varphi)$-stable coloring $\varphi''$. However,   $\pbar''(x)=\alpha+1$ and $\pbar''(y)=1$, showing    
	a contradiction to Claim~\ref{u-2-out-neighbor}. 
As  $w_t$ and $x$ are $(1,\alpha+1)$-linked, we do a sequence of Kempe changes around $r$ from $w_k^*$ to $w_1^*$ as below:
	let $\varphi_0=\varphi$ and $\lambda_{k+1}=\alpha+1$, 
	$$\varphi_j=\varphi_{j-1}/P_{w^*_{k-(j-1)}}(1,\lambda_{k+1-(j-1)}, \varphi_{j-1}) \quad \text{for each $j\in [1,k]$}.$$
	Note that $$P_{r}(1,\lambda_{k-(j-1)},\varphi_j)=rw^*_{k-(j-1)} \quad \text{for each $j\in [1,k]$}, $$
and that $\varphi_k$ is $(F,\varphi)$-stable, $\varphi_k(ru)=\varphi(ru)$,
	$\varphi_k(ux)=\varphi(ux)$,  and $\pbar_k(x)=\pbar(x)=\alpha+1$, 
	but $z$ and $r$ are $(1,\lambda)$-unlinked with respect to $\varphi_k$. 
	Now doing a  $(1,\lambda)$-swap at $z$ gives a contradiction to Claim~\ref{u-2-out-neighbor}.
	This finishes the proof of Theorem~\ref{Thm:vizing-fan}. 
\end{proof}

\section{Proof of Theorem~\ref{Thm:adj_small_vertex}}

\begin{THM2}
	If $G$ is an HZ-graph with maximum degree $\Delta\ge 4$,  then for any two  adjacent vertices $x, y\in V_{\Delta-1}$, $N_\Delta(x)=N_\Delta(y)$.
\end{THM2}
\begin{proof}
	Assume to the contrary that  $N_\Delta(x)\ne N_\Delta(y)$. Then there exists a vertex $r\in N_\Delta(x)\setminus N_\Delta(y)$. 
	Equivalently,  $x\in N_{\Delta-1}(r)$ and  $y\not\in N_{\Delta-1}(r)$. 
	By Theorem~\ref{Thm:vizing-fan2b} \eqref{ele2},  let $s_1\in N_{\Delta-1}(r)$
	and $\varphi\in \CC^\Delta(G-rs_1)$, and  $F=F_\varphi(r,s_1: s_\alpha)$ be the typical 2-inducing multifan 
	such that either $V(F)=N_{\Delta-1}[r]$ or  $F$ is contained in
	a pseudo-multifan  $S$ with   $V(S)=N_{\Delta-1}[r]$. 
	Let $ N_{\Delta-1}(r)=\{s_1,\ldots, s_{\Delta-2}\}$.
	We consider two cases according to  if $x\in V(F)$
	to finish the proof.

	Assume first that  $x\notin V(F)$. This implies that $V(F)\ne N_{\Delta-1}[r]$. 
	Applying  Theorem~\ref{Thm:vizing-fan2b} \eqref{ele2}, it then follows that  
	$N_{\Delta-1}[r]$ is  the vertex set of a typical 2-inducing pseudo-multifan.  Let $\pbar(x)=\delta$ and $\pbar(y)=\lambda$. Since $V(S)=N_{\Delta-1}[r]$ is $\varphi$-elementary, $ \delta,\lambda\in \pbar(S)$. 
	By Lemma~\ref{thm:vizing-fan1} \eqref{thm:vizing-fan1b}
	or Lemma~\ref{pseudo-fan-ele} \eqref{pseudo-a1}, we know
	that $\pbar^{-1}_S(\lambda)$ and $r$ are $(1,\lambda)$-linked and $x$ and $r$ are $(1,\delta)$-linked. 
	By doing a $(\lambda,1)-(1,\delta)$-swap at $y$, we find $(S,\varphi)$-stable $\varphi'\in \CC^\Delta(G-rs_1)$ such that $\pbar'(y)=\delta$. 
	Let $\varphi'(xy)=\tau$.  Then $P_x(\delta,\tau,\varphi')=xy$, showing a contradiction to 
	Lemma~\ref{pseudo-fan-ele}~\eqref{pseudo-b} or \eqref{pseudo-c} depending on $\tau \in \pbar'(F)$ or $\tau \in \pbar'(S)\setminus \pbar'(F)$.

	Assume then that $x\in V(F)$. We claim that we may assume $x=s_1$. 
	Let $x=s_i$ for some $i\in[1,\alpha]$, and $\varphi'$ be 
	obtained from $\varphi$  by shift from $s_2$ to $s_{i-1}$, uncoloring $rs_i$, 
	and coloring $rs_1$ by 2.   
	The sequence 
	$F^*=(r,rs_i, s_i,rs_{i+1}, s_{i+1}, \ldots, rs_\alpha, s_\alpha, rs_{i-1}, s_{i-1}, \ldots, rs_1,s_1)$ is  a multifan with respect to $\varphi'$.    Since 
	the shift and ``changing'' the uncolored
	edge operation like above is revertible, and   $V(S)$ and $\pbar(S)$
	are kept unchanged under such an operation,  
 we conclude that $N_{\Delta-1}[r]$ is still  the vertex set of a pseudo-multifan.
	By permuting the name of 
	colors and the label of the vertices in $N_{\Delta-1}(r)$, 
	we may assume that  $x=s_1$. 
	Still denote the current coloring by $\varphi$, the multifan by $F$,
	and the pseudo-multifan by $S$.

	 By doing a $(1,\pbar(y))$-swap at $y$, 
	we assume  $\pbar(y)=1$. 
	Let 	$\varphi(s_1y)=\tau$.  By exchanging the roles of the color 2 and $\Delta$ 
	if necessary, we may assume that 
	$\varphi(s_1y)$ is either a 2-inducing color of $F$ or is a color from $\pbar(S)\setminus \pbar(F)$. Let  $\varphi'=\varphi/P_y(1,\Delta,\varphi)$. 
	Now $P_{s_1}(\tau,\Delta, \varphi')=s_1y$.
	This gives a contradiction to Lemma~\ref{thm:vizing-fan2}~\eqref{thm:vizing-fan2-b}
	that $s_1$ and $\pbar'^{-1}_F(\tau)$  are  $(\tau, \Delta)$-linked if  $\tau$ is 2-inducing, 
	and gives a contradiction to Lemma~\ref{pseudo-fan-ele}~\eqref{pseudo-b} that $s_1$ and $\pbar'^{-1}_S(\tau)$
	are $(\tau,\Delta)$-linked if 
	$\tau \in \pbar(S)\setminus \pbar(F)$. 
\end{proof}

\section{Proof of Theorem~\ref{Thm:nonadj_Delta_vertex}}

\begin{THM3}
	Let $G$ be an HZ-graph with maximum degree $\Delta\ge 7$ and $u, r\in V_{\Delta}$.
	If $N_{\Delta-1}(u)\ne N_{\Delta-1}(r)$ and $N_{\Delta-1}(u)\cap  N_{\Delta-1}(r)\ne \emptyset$,
	then $|N_{\Delta-1}(u)\cap  N_{\Delta-1}(r)|=\Delta-3$, i.e. $|N_{\Delta-1}(u)\setminus N_{\Delta-1}(r)|=|N_{\Delta-1}(r)\setminus  N_{\Delta-1}(u)|=1$. 
\end{THM3}	

\begin{proof}
	Assume to the contrary that there exist $u,r\in N_{\Delta}$
	such that $1\le |N_{\Delta-1}(r)\cap N_{\Delta-1}(u)|\le \Delta-4$. 
	By Theorem~\ref{Thm:vizing-fan2b} \eqref{ele2}, there exist $s_1\in N_{\Delta-1}(r)$
	and $\varphi\in \CC^\Delta(G-rs_1)$ such that $N_{\Delta-1}[r]$
	is  the vertex set of a   typical  2-inducing pseudo-multifan. By this assumption of being typical, 
	we have $ N_{\Delta-1}(r)=\{s_1,\ldots, s_{\Delta-2}\}$, 
	$\pbar(r)=1$, and $\pbar(s_1)=\{2,\Delta\}$. 
	Let 
	$x,y\in N_{\Delta-1}(u)\setminus N_{\Delta-1}(r)$ be two distinct vertices, and  
	  $S:=S_\varphi(r, s_1:s_\alpha: s_{\Delta-2})$ be this pseudo-multifan with $F_\varphi(r,s_1:s_\alpha)$ being the typical 2-inducing multifan  contained in $S$. Since $V(S)=N_{\Delta-1}[r]$ and $V(S)$
	is $\varphi$-elementary, it follows that $\pbar(S)=[1,\Delta]$. We consider two cases.

	\smallskip
	\noindent{\bf Case 1}: $V(S) \ne V(F)$. 
	\smallskip

	In this case, we will repeatedly apply Lemma~\ref{pseudo-fan-ele}~\eqref{pseudo-a1}, \eqref{pseudo-b} or \eqref{pseudo-c}.  
	Assume first that for each $i\in [1,\alpha]$,  $s_i\notin N_{\Delta-1}(u)\cap N_{\Delta-1}(r)$.
	Then by $N_{\Delta-1}(r)\cap N_{\Delta-1}(u)\ne \emptyset$,  there exists $w_1\in \{s_{\alpha+1}, \ldots, s_{\Delta-2}\}$
	such that $w_1\in N_{\Delta-1}(u)\cap N_{\Delta-1}(r)$. 
	Let $\varphi(rw_1)=\delta_1$ and $ \pbar(w_1)=\delta_2$.
	Note that $\delta_1,\delta_2\in \pbar(S)\setminus \pbar(F)$.  
	We claim  that we may assume $\varphi(ux)=\delta_2$. Otherwise, let $\varphi(ux)=\delta^*\ne \delta_2$. By Lemma~\ref{pseudo-fan-ele}~\eqref{pseudo-a1}, \eqref{pseudo-b} or \eqref{pseudo-c} depending on what $\pbar(x)$ is, we can do a $(\pbar(x),\delta_2)-(\delta_2, \delta^*)$-swap 
at $x$ in getting an $(S,\varphi)$-stable coloring, still call it $\varphi$ such that $\varphi(ux)=\delta_2$. 
	Let $\varphi(w_1u)=\tau$
	and  $\varphi'$ be obtained from $\varphi$
	by doing a $(\pbar(x),\delta_1)-(\delta_1,\tau)$-swap at 
	$x$. 
By Lemma~\ref{pseudo-fan-ele}~\eqref{pseudo-a1}, \eqref{pseudo-b} or \eqref{pseudo-c}, $\varphi'$ is $(S,\varphi)$-stable 
such that $\varphi'(w_1u)=\varphi(w_1u)=\tau$ and $\pbar'(x)=\tau$. 
	However,  $P_{w_1}(\delta_2, \tau, \varphi')=w_1ux=P_x(\delta_2, \tau, \varphi')$, showing a contradiction to 
	Lemma~\ref{pseudo-fan-ele}~\eqref{pseudo-a1}, \eqref{pseudo-b} or \eqref{pseudo-c} (depending on if $\tau=1$, $\tau \in \pbar(F)\setminus \{1\}$ or $\tau \in \pbar(S)\setminus \pbar(F)$)
	that $w_1$ and $\pbar^{-1}_S(\tau)$ are $(\delta_2,\tau)$-linked with respect to $\varphi'$.

	Assume now that  there exists $s_i\in N_{\Delta-1}(u)\cap N_{\Delta-1}(r)$ for some $i\in [1,\alpha]$. 
		By shift from $s_2$
	to $s_{i-1}$, uncoloring  $rs_i$, and coloring  $rs_1$ by 2, 
	we obtain a new multifan $F^*=(r,rs_i,s_i, rs_{i+1}, s_{i+1}, \ldots, rs_\alpha, s_\alpha, rs_{i-1}, s_{i-1}, \ldots, rs_1,s_1)$. 
	By  permuting the name of 
	colors and the label of the vertices in $N_{\Delta-1}(r)$ such that $i+1$ is permuted to $2$
	and $s_i$ is renamed as $s_1$, 
	we  assume that  $s_1\in N_{\Delta-1}(u)\cap N_{\Delta-1}(r)$ and $F^*$ is a  typical multifan. 
	
	Recall that $x\in N_{\Delta-1}(u)\setminus N_{\Delta-1}(r)$.
	Let $\pbar(x)=\lambda$.  By Lemma~\ref{thm:vizing-fan1} \eqref{thm:vizing-fan1b}
	or Lemma~\ref{pseudo-fan-ele} \eqref{pseudo-a1}, we know
	that $\pbar^{-1}_S(\lambda)$ and $r$ are $(1,\lambda)$-linked. 
	By doing a $(1,\lambda)$-swap at $x$ if necessary, we assume 
	$\pbar(x)=1$.  By exchanging the roles of  the colors 2 and $\Delta$, we 
	assume that $\varphi(s_1u)$  equals 1, or is a 2-inducing color of $F$, or is  a color from $\pbar(S)\setminus \pbar(F)$. Note that by Lemma~\ref{pseudo-fan-ele}~\eqref{pseudo-b}, 
	for a color $\delta \in \pbar(S)\setminus \pbar(F)$, and for any color $\tau\in \pbar(F)$, 
	$\pbar_S^{-1}(\delta)$ and $\pbar_S^{-1}(\tau)$
	are $(\delta, \tau)$-linked and $r\in P_{\pbar_S^{-1}(\delta)}(\delta,\tau,\varphi)$.
	
	Let $\varphi(ux)=\tau$. If $\tau$ is a 2-inducing color of $F$ or  is from $\pbar(S)\setminus \pbar(F)$, we  do  $(1,\Delta)-(\Delta,\tau)-(\tau,1)$-swaps at $x$. If $\tau$ is a $\Delta$-inducing color of $F$, let $\delta\in \pbar(S)\setminus \pbar(F)$,   we  do  $(1,\delta)-(\delta,\tau)-(\tau,1)-(1,\Delta)-(\Delta,\delta)-(\delta,1)$-swaps at $x$. In both cases, 
	we let $\varphi'$ be the resulting  coloring. We have  $\varphi'(ux)=\Delta$ 
	and $\pbar'(x)=1$. 
	Since $\varphi(s_1u) \ne \Delta, \tau$,  still  $\varphi'(s_1u)$ 
	equals 1, or is a 2-inducing color of $F$,  or is from $\pbar(S)\setminus \pbar(F)$. 
	
	Let $\varphi'(s_1u)=\gamma$. 
	Since $s_1$ and $r$ are $(1,\Delta)$-linked with respect to $\varphi'$, 
	$\gamma\ne 1$.  Thus,   $\gamma$  is a 2-inducing color of $F$,  or is from $\pbar(S)\setminus \pbar(F)$.  By Lemma~\ref{thm:vizing-fan2} \eqref{thm:vizing-fan2-a} or Lemma~\ref{pseudo-fan-ele}~\eqref{pseudo-b},  $u\in P_x(1,\gamma,\varphi')$ (otherwise, $s_1$ and $x$ are $(\gamma,\Delta)$-linked after a $(1,\gamma)$-swap at $x$).  Let $\varphi''=\varphi'/P_x(1,\gamma,\varphi')$.
	 Now  $\varphi''(s_1u)=1$, $\pbar''(x)=\gamma$, and $K=(r, rs_1,s_1, s_1u,  u, ux, x)$ is a Kierstead path with respect to $rs_1$ and $\varphi''$. Let $\delta\in \pbar''(S)\setminus \pbar''(F)$.  If $\gamma \in \pbar''(S)\setminus \pbar''(F)$,  we do nothing. Otherwise, 
	we do a $(\gamma,\delta)$-swap at $x$ (by Lemma~\ref{pseudo-fan-ele}~\eqref{pseudo-b}, this swap does not end at any vertex of $S$).  Denote by $\varphi'''$ the resulting coloring. 
	Since $d_G(s_1)=\Delta-1$,  in both cases, 
	by Lemma~\ref{Lemma:kierstead path1}~(b), $x$ and $s_1$ are  $(2,\pbar'''(x))$-linked.
	Since $\pbar'''(x) \in \pbar'''(S)\setminus \pbar''(F)$,  
	we achieve a contradiction to Lemma~\ref{pseudo-fan-ele}~\eqref{pseudo-b}.

	\smallskip 
	\noindent{\bf Case 2}: $V(S)=V(F)$. 
	\smallskip 
	
	We claim that we may choose $s_1$ such that  $s_1\in N_{\Delta-1}(u)\cap N_{\Delta-1}(r)$. 
	If $s_1\in N_{\Delta-1}(u)\cap N_{\Delta-1}(r)$, 
	then we are done. Otherwise, let $s_i\in N_{\Delta-1}(u)\cap N_{\Delta-1}(r)$. 
	We shift from $s_2$
	to $s_{i-1}$,  uncolor  $rs_i$ and color  $rs_1$ by 2. By relabeling 
	colors and vertices, 
	we may assume that  $s_1\in N_{\Delta-1}(u)\cap N_{\Delta-1}(r)$ and $F^*=(r, rs_i, s_i, rs_{i+1}, s_{i+1}, \ldots, rs_\alpha, s_\alpha, rs_{i-1}, s_{i-1}, \ldots, rs_1, s_1)$ is a  typical multifan.  
	We  let $F_\varphi(r,s_1: s_\alpha:s_{\Delta-2})$ be such a typical multifan.
	
	For a coloring $\psi\in \CC^\Delta(G-rs_1)$, if $\pbar(s_1)=\overline{\psi}(s_1)$, $\pbar(F)=\overline{\psi}(F)$, 
	and some permutation of $F$ is still a multifan with respect to  $rs_1$ and 
	$\psi$, we call $\psi$ a \emph{near $(F,\varphi)$-stable} coloring.  As only colors in $\pbar(s_1)$ will be essential for the proof, we   will not distinguish between $\varphi$ and any near $(F,\varphi)$-stable coloring.  As the vertex set of all the resulting multifans 
	 is always $N_{\Delta-1}[r]$,  for a color $\alpha\in [1,\Delta]$, we use $\overline{\psi}^{-1}(\alpha)$
	 to denote the vertex from  $N_{\Delta-1}[r]$ that misses $\alpha$ with respect to $\psi$.

Let $\psi\in \CC^\Delta(G-rs_1)$ be  near $(F,\varphi)$-stable  and $F^*$ be the corresponding multifan.  	
	The following two facts will be used frequently in the proof without mentioning.  
\begin{enumerate}
	\item[Fact 1]For any $i\in [2,\Delta]$,  since $r$ and $\overline{\psi}^{-1}(i)$ are $(1,i)$-linked by Lemma~\ref{thm:vizing-fan1}~\eqref{thm:vizing-fan1b},  	doing a $(1,i)$-swap at vertices outside of $V(F^*)$ gives an $(F^*,\psi)$-stable  and so a near
	$(F,\varphi)$-stable coloring.  
	\item[Fact 2] For any $2$-inducing color $\tau$ 
	and $\Delta$-inducing color  $\delta$  of $F^*$,   $\overline{\psi}^{-1}(\tau)$ and $\overline{\psi}^{-1}(\delta)$ are $(\tau,\delta)$-linked by Lemma~\ref{thm:vizing-fan2}~\eqref{thm:vizing-fan2-a}. Thus doing a  $(\tau,\delta)$-swap at a 
	vertex outside of $V(F^*)$ or, when $\tau \ne 2$ and $\delta \ne \Delta$, doing a  $(\tau,\delta)$-swap at $\overline{\psi}^{-1}(\tau)$ 
	if $r\not\in P_{\overline{\psi}^{-1}(\tau)}(\tau,\delta,\psi)$  gives a near $(F^*,\psi)$-stable and so a near
	$(F,\varphi)$-stable coloring.  
\end{enumerate}
We denote by $S(u;s_1, x,y)$ the star subgraph of $G$
that is centered at $u$ consisting of edges $us_1, ux$, and $uy$.
Recall that $x,y\in N_{\Delta-1}(u)\setminus N_{\Delta-1}(r)$ are distinct vertices.

	\begin{CLA}\label{x-y-missing-colors}
		We may assume that $\pbar(x)=2$ and $\pbar(y)=\Delta$ 
		or $\pbar(x)=\pbar(y)=\Delta$. 
	\end{CLA}
	
	\proof[Proof of Claim~\ref{x-y-missing-colors}]
	
	 By doing  $(\pbar(x),1)-(1,2)$-swaps at $x$, 
	we find $(F,\varphi)$-stable $\varphi'\in \CC^\Delta(G-rs_1)$ such that $\pbar'(x)=2$. 
	Now, let $\pbar'(y)=\lambda$. If $\lambda=2$, then doing $(2,1)-(1,\Delta)$-swaps at both $x$
	and $y$, we find $(F,\varphi')$-stable $\varphi''\in \CC^\Delta(G-rs_1)$ such that   $\pbar''(x)=\pbar''(y)=\Delta$. If  $\lambda\ne 2$, 
	by doing $(\lambda,1)-(1,\Delta)$-swaps at $y$, 
	 we find $(F,\varphi')$-stable $\varphi''\in \CC^\Delta(G-rs_1)$ such that  $\pbar''(x)=2$ and $\pbar''(y)=\Delta$. As $\varphi''$ is $(F,\varphi')$-stable and $\varphi'$ is $(F,\varphi)$-stable, it  follows that $\varphi''$ is $(F,\varphi)$-stable. So we can take $\varphi''$
	 to be $\varphi$. 
	\qed 
	
	By Claim~\ref{x-y-missing-colors}, we assume $\pbar(x)=2$ and $\pbar(y)=\Delta$ 
	or $\pbar(x)=\pbar(y)=\Delta$ and so consider two cases below. 
	
	\smallskip 
		{\noindent\bf \setword{Subcase 2.1}{Case 1}: $\pbar(x)=2$ and $\pbar(y)=\Delta$.}
	\smallskip 

	Note that by doing first a $(2,1)$-swap at $x$, then a $(1,\Delta)$-swap at both $x$ and $y$, and finally a 
	$(1,2)$-swap at $y$, we can always identify this current case with the case that $\pbar(x)=\Delta$ and $\pbar(y)=2$.
	Let $\varphi(ux)=\tau$ and $\varphi(uy)=\lambda$.  By exchanging the roles of 
	the two colors 2 and $\Delta$,  we consider two cases below: (A) $\varphi(uy)=\lambda=1$;
	and (B) $\varphi(uy)=\lambda$ is 2-inducing. 
	(When $\varphi(uy)$ is $\Delta$-inducing, by  assuming  $\pbar(x)=\Delta$ and $\pbar(y)=2$,   the argument will be symmetric to the argument for case (B) above.)
	
	In both cases of (A) and (B), we do  $(\Delta,\lambda)-(\lambda,1)$-swaps at $y$ and   call the resulting coloring $\varphi_1$ and the resulting multifan  $F_1$.  Note that $\varphi_1$ is near $(F,\varphi)$-stable. 		Let $\varphi_1(s_1u)=\delta$.   
	The current coloring on $S(u; s_1, x, y)$ is as shown in $J_1$  of  Figure~\ref{J1}.  
	
	\begin{CLA}\label{claim1}
	The color $\varphi_1$ can be modified into a near $(F_1,\varphi_1)$-stable coloring such that 
	the color on 
	$S(u; s_1, x, y)$   is as  in $J_2$ of Figure~\ref{J1}. 
	\end{CLA}
		
	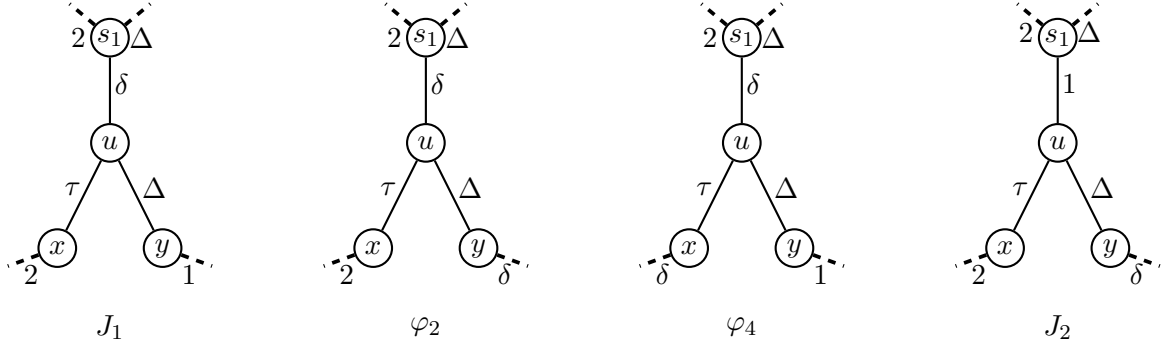
\begin{figure}[!htb]
		\begin{center}
			\begin{tikzpicture}[scale=0.7]
			
			{\tikzstyle{every node}=[draw ,circle,fill=white, minimum size=0.5cm,
				inner sep=0pt]
				\draw[black,thick](0,-3) node (s1)  {$s_1$};
				\draw [black,thick](0, -5) node (u)  {$u$};
				\draw [black,thick](-1, -7) node (x)  {$x$};
				\draw [black,thick](1, -7) node (y)  {$y$};
			}
			\path[draw,thick,black]
			
			(s1) edge node[name=la,pos=0.4] {\color{black}\,\,\, $\delta$} (u)
			(u) edge node[name=la,pos=0.4] {\color{black}$\tau$\quad\quad} (x)
			(u) edge node[name=la,pos=0.4] {\color{black}\,\,\,\,\,\,\, $\Delta$} (y);
			
			
			\draw[dashed, black, line width=0.5mm] (s1)--++(140:1cm);
			\draw[dashed, black, line width=0.5mm] (s1)--++(40:1cm); 
			\draw[dashed, black, line width=0.5mm] (x)--++(200:1cm); 
			\draw[dashed, black, line width=0.5mm] (y)--++(340:1cm);
			\draw[black] (-0.6, -3) node {$2$};  
			\draw[black] (0.6, -3) node {$\Delta$};  
			\draw[black] (-1.5, -7.5) node {$2$}; 
			\draw[black] (1.5, -7.5) node {$1$};
			
			\draw[black] (0, -8.5) node {$J_1$}; 
			
			\begin{scope}[shift={(6,0)}]
			
			{\tikzstyle{every node}=[draw ,circle,fill=white, minimum size=0.5cm,
				inner sep=0pt]
				\draw[black,thick](0,-3) node (s1)  {$s_1$};
				\draw [black,thick](0, -5) node (u)  {$u$};
				\draw [black,thick](-1, -7) node (x)  {$x$};
				\draw [black,thick](1, -7) node (y)  {$y$};
			}
			\path[draw,thick,black]
			
			(s1) edge node[name=la,pos=0.4] {\color{black}\,\,\, $\delta$} (u)
			(u) edge node[name=la,pos=0.4] {\color{black}$\tau$\quad\quad} (x)
			(u) edge node[name=la,pos=0.4] {\color{black}\,\,\,\,\,\,\, $\Delta$} (y);
			
			
			\draw[dashed, black, line width=0.5mm] (s1)--++(140:1cm);
			\draw[dashed, black, line width=0.5mm] (s1)--++(40:1cm); 
			\draw[dashed, black, line width=0.5mm] (x)--++(200:1cm); 
			\draw[dashed, black, line width=0.5mm] (y)--++(340:1cm);
			\draw[black] (-0.6, -3) node {$2$};  
			\draw[black] (0.6, -3) node {$\Delta$};  
			\draw[black] (-1.5, -7.5) node {$2$}; 
			\draw[black] (1.5, -7.5) node {$\delta$};
			
			\draw[black] (0, -8.5) node {$\varphi_2$}; 
			\end{scope}

				\begin{scope}[shift={(12,0)}]
			
			{\tikzstyle{every node}=[draw ,circle,fill=white, minimum size=0.5cm,
				inner sep=0pt]
				\draw[black,thick](0,-3) node (s1)  {$s_1$};
				\draw [black,thick](0, -5) node (u)  {$u$};
				\draw [black,thick](-1, -7) node (x)  {$x$};
				\draw [black,thick](1, -7) node (y)  {$y$};
			}
			\path[draw,thick,black]
			
			(s1) edge node[name=la,pos=0.4] {\color{black}\,\,\, $\delta$} (u)
			(u) edge node[name=la,pos=0.4] {\color{black}$\tau$\quad\quad} (x)
			(u) edge node[name=la,pos=0.4] {\color{black}\,\,\,\,\,\,\, $\Delta$} (y);
			
			
			\draw[dashed, black, line width=0.5mm] (s1)--++(140:1cm);
			\draw[dashed, black, line width=0.5mm] (s1)--++(40:1cm); 
			\draw[dashed, black, line width=0.5mm] (x)--++(200:1cm); 
			\draw[dashed, black, line width=0.5mm] (y)--++(340:1cm);
			\draw[black] (-0.6, -3) node {$2$};  
			\draw[black] (0.6, -3) node {$\Delta$};  
			\draw[black] (-1.5, -7.5) node {$\delta$}; 
			\draw[black] (1.5, -7.5) node {$1$};
			
			\draw[black] (0, -8.5) node {$\varphi_4$}; 
			\end{scope}
			\begin{scope}[shift={(18,0)}]

			{\tikzstyle{every node}=[draw ,circle,fill=white, minimum size=0.5cm,
				inner sep=0pt]
				\draw[black,thick](0,-3) node (s1)  {$s_1$};
				\draw [black,thick](0, -5) node (u)  {$u$};
				\draw [black,thick](-1, -7) node (x)  {$x$};
				\draw [black,thick](1, -7) node (y)  {$y$};
			}
			\path[draw,thick,black]
			
			(s1) edge node[name=la,pos=0.4] {\color{black}\,\,\, $1$} (u)
			(u) edge node[name=la,pos=0.4] {\color{black}$\tau$\quad\quad} (x)
			(u) edge node[name=la,pos=0.4] {\color{black}\,\,\,\,\,\,\, $\Delta$} (y);
			
			
			\draw[dashed, black, line width=0.5mm] (s1)--++(140:1cm);
			\draw[dashed, black, line width=0.5mm] (s1)--++(40:1cm); 
			\draw[dashed, black, line width=0.5mm] (x)--++(200:1cm); 
			\draw[dashed, black, line width=0.5mm] (y)--++(340:1cm);
			\draw[black] (-0.6, -2.9) node {$2$};  
			\draw[black] (0.6, -2.9) node {$\Delta$};  
			\draw[black] (-1.5, -7.5) node {$2$}; 
			\draw[black] (1.5, -7.5) node {$\delta$};
			\draw[black] (0, -8.5) node {$J_2$}; 
			\end{scope}
			\end{tikzpicture}
		\end{center}
	\vspace{-0.7cm}
		\caption{Coloring of $S(u; s_1, x, y)$}
		\label{J1}
	\end{figure}

	\proof[Proof of Claim~\ref{claim1}]
		Since $s_1$ and $r$ are $(1,\Delta)$-linked by Lemma~\ref{thm:vizing-fan1}~\eqref{thm:vizing-fan1b}, 
	we know  $\varphi(s_1u)=\delta\ne 1$. 
	If $u\in P_y(1,\delta,\varphi_1)$, then a $(1,\delta)$-swap at $y$ gives $J_2$. Thus, we assume 
	$u\not\in P_y(1,\delta,\varphi_1)$. This implies that  $\delta$ is $\Delta$-inducing with respect to  $F_1$ and $\varphi_1$. (Otherwise, after a $(1,\delta)$-swap at $y$, $s_1$ and $\pbar_1^{-1}(\delta)$ are $(\delta,\Delta)$-unlinked, showing a contradiction to Lemma~\ref{thm:vizing-fan2} \eqref{thm:vizing-fan2-a}.)
	
	Let $\varphi_2=\varphi_1/P_y(1,\delta,\varphi_1)$ (see Figure~\ref{J1}). 
	We claim that $\tau$    is $2$-inducing with respect to $F_1$ and $\varphi_2$.  
	Otherwise $\tau$   is  1 or is $\Delta$-inducing with respect to $F_1$ and $\varphi_2$.  We do  $(2,\tau)-(\tau,1)$-swaps at $x$ and call the resulting coloring $\varphi_2'$ and the resulting multifan $F_1'$. 
	Again, as $P_{s_1}(\delta,\Delta, \varphi_2')=s_1uy$, $\delta$ is still a $\Delta$-inducing color of $F_1'$ with respect to $\varphi_2'$ by Lemma~\ref{thm:vizing-fan2}~\eqref{thm:vizing-fan2-a}. 
	Since $\varphi_2'(ux)=2$,  we must have $u\in P_x(1,\delta, \varphi_2')$: 
	otherwise, after a $(1,\delta)$-swap at $x$, $s_1$ and $x$ are $(2,\delta)$-linked with respect to the current coloring, 
	contradicting Lemma~\ref{thm:vizing-fan2}~\eqref{thm:vizing-fan2-a}. 
	Now, let $\varphi_2^*$ be obtained from  $\varphi_2'$ by doing a $(1,\delta)$-swap at both $x$
	and $y$. We get  $P_{s_1}(1,\Delta,\varphi_2^*)=s_1uy$, showing a contradiction to 
	Lemma~\ref{thm:vizing-fan1}~\eqref{thm:vizing-fan1b} that $s_1$ and $r$
	are $(1,\Delta)$-linked with respect to $\varphi_2^*$. 
	
	Thus   $\tau$   is 2-inducing with respect to $F_1$ and $\varphi_2$.  
	First, we let $\varphi_3=\varphi_2/P_x(1,2,\varphi_2)$. 
	Note that  $u\not\in P_x(1,\delta, \varphi_3)$ and $u\not\in P_y(1,\delta, \varphi_3)$.  	Since otherwise, after a $(1,\delta)$-swap at both $x$ and $y$,  $s_1$ and $y$ are $(1,\Delta)$-linked with respect to the current coloring,  showing a contradiction to 
	Lemma~\ref{thm:vizing-fan1}~\eqref{thm:vizing-fan1b} that $s_1$ and $r$
	are $(1,\Delta)$-linked. Since $\delta$ is $\Delta$-inducing and $\tau$
	is 2-inducing with respect to $F_1$ and $\varphi_3$,  $\pbar_3^{-1}(\delta)$ and $\pbar_3^{-1}(\tau)$
	are $(\delta,\tau)$-linked  by Lemma~\ref{thm:vizing-fan2} \eqref{thm:vizing-fan2-a}.
	Then we let $\varphi_4$ be obtained from $\varphi_3$ by doing a $(1,\delta)$-swap at both $x$ and $y$ (see Figure~\ref{J1}), and doing  a
	$(\tau,\delta)$-swap at $\pbar_3^{-1}(\delta)$ (and so also at $\pbar_3^{-1}(\tau)$).  Since $\varphi_4$ is  near $(F_1,\varphi_3)$ stable,
	we let $F_2$ be the resulting multifan. 
Note that $\delta$
	is  a 2-inducing color  and $\tau$ is   a $\Delta$-inducing color of $F_2$ with respect to $\varphi_4$. As a consequence, $u\in P_y(1,\delta, \varphi_4)$.  Since otherwise, after a $(1,\delta)$-swap at $y$,  
	$s_1$ and $y$ are $(\delta,\Delta)$-linked, 
	contradicting Lemma~\ref{thm:vizing-fan2}~\eqref{thm:vizing-fan2-a}.   We then let $\varphi_5$ be obtained from $\varphi_4$
	by doing a
	$(1,\delta)$-swap at both  $x$, $y$ (and so also $u$), and then a $(1,2)$-swap at $x$. 
	 We obtain  the desired coloring on  $S(u; s_1, x, y)$. 
	 \qed

	By Claim~\ref{claim1}, we let $\varphi_2\in \CC^\Delta(G-rs_1)$ be a near $(F_1,\varphi_1)$-stable coloring and $F_2$ be a corresponding multifan such that under $\varphi_2$,  the color on  $S(u; s_1, x, y)$  is 
	as  in Figure~\ref{J1} $J_2$. 
		Now $K=(r,rs_1,s_1,s_1u, u,uy,y)$ is a Kierstead path with respect to $rs_1$ 
	and $\varphi_2$. 
	Since $d_G(s_1)=\Delta-1$, by Lemma~\ref{Lemma:kierstead path1} (b),  $y$ and $s_1$ are $(2,\delta)$-linked. This implies that 
	$\delta$ must be a 2-inducing color of $F_2$, as otherwise, 
	$s_1$ and $\pbar_2^{-1}(\delta)$ should be $(2,\delta)$-linked. 
	If $\tau$ is $\Delta$-inducing of $F_2$, then as $\pbar_2^{-1}(\delta)$ and $\pbar_2^{-1}(\tau)$
	are $(\delta,\tau)$-linked by Lemma~\ref{thm:vizing-fan2} \eqref{thm:vizing-fan2-a},  
	we  do a $(\delta,\tau)$-swap at $y$. Again by  Lemma~\ref{Lemma:kierstead path1} (b),
	$y$ and $s_1$ are $(2,\tau)$-linked, showing a contradiction to 
	Lemma~\ref{thm:vizing-fan2} \eqref{thm:vizing-fan2-a} 
that $\pbar_2^{-1}(\tau)$ and $s_1$
	are $(2,\tau)$-linked. 
	Therefore, $\tau$ is a $2$-inducing color of $F_2$. We first do $(2,1)-(1,\Delta)$ swaps at $x$, and let $\varphi_3$ be the resulting coloring (see Figure~\ref{J3}).  At this step, $\varphi_3(s_1u)=1$, $\varphi_3(uy)=\Delta$, 
	$\pbar_3(y)=\delta$, and  
	$y$ and $s_1$ are $(2,\delta)$-linked with respect to $\varphi_3$ by Lemma~\ref{Lemma:kierstead path1} (b).  
	Call this fact $(*)$.
	
	Let $\varphi_4=\varphi_3/P_x(\tau,\Delta,\varphi_3)$ (see Figure~\ref{J3}) and $F_3$ be the resulting multifan.  Since $\varphi_3^{-1}(\tau)$ appears before the edge with color $\tau$ in $F_2$,  $\tau$ is still a  2-inducing color of 
	$F_3$.
	As $s_1$ and $r$
	are $(1,\Delta)$-linked by Lemma~\ref{thm:vizing-fan1}~\eqref{thm:vizing-fan1b}, 
	we have $u\in P_x(1,\tau,\varphi_4)$.  Let $\varphi_5=\varphi_4/P_x(1,\tau,\varphi_4)$.  
	The coloring of $S(u;s_1,x,y)$
	is now as in Figure~\ref{J3} $J_3$.
	Since $2,\delta \not\in \{1,\tau, \Delta\}$, $y$ and $s_1$ are still $(2,\delta)$-linked with respect to 
	$\varphi_5$ by fact $(*)$, which further implies that $\delta$ is  a 2-inducing color of $F_3$ 
	with respect to $\varphi_5$. Since $\varphi_5$ is $(F_3,\varphi_4)$-stable, $\tau$ is still a 2-inducing colors of 
	$F_3$ with respect to $\varphi_5$. 
	We consider two cases to finish the remaining part of the proof. 
	
	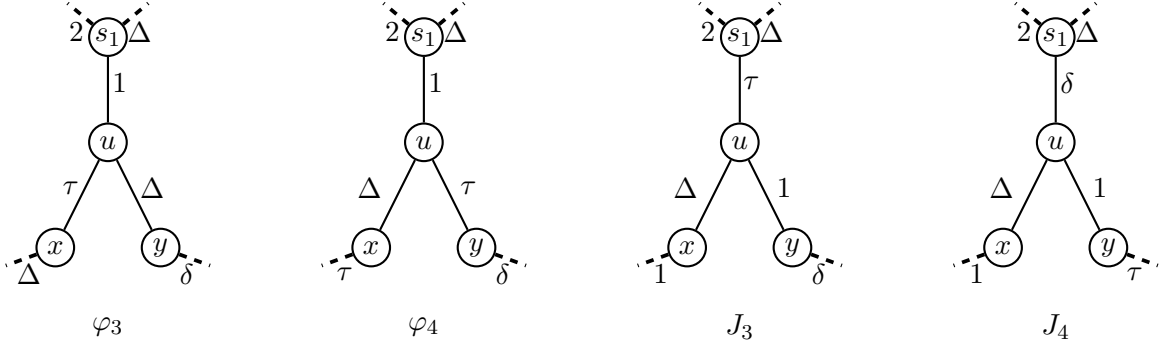
\begin{figure}[!htb]
		\begin{center}
			\begin{tikzpicture}[scale=0.7]

			{\tikzstyle{every node}=[draw ,circle,fill=white, minimum size=0.5cm,
				inner sep=0pt]
				\draw[black,thick](0,-3) node (s1)  {$s_1$};
				\draw [black,thick](0, -5) node (u)  {$u$};
				\draw [black,thick](-1, -7) node (x)  {$x$};
				\draw [black,thick](1, -7) node (y)  {$y$};
			}
			\path[draw,thick,black]
			
			(s1) edge node[name=la,pos=0.4] {\color{black}\,\,\, $1$} (u)
			(u) edge node[name=la,pos=0.4] {\color{black}$\tau$\quad\quad} (x)
			(u) edge node[name=la,pos=0.4] {\color{black}\,\,\,\,\,\,\, $\Delta$} (y);
			
			
			\draw[dashed, black, line width=0.5mm] (s1)--++(140:1cm);
			\draw[dashed, black, line width=0.5mm] (s1)--++(40:1cm); 
			\draw[dashed, black, line width=0.5mm] (x)--++(200:1cm); 
			\draw[dashed, black, line width=0.5mm] (y)--++(340:1cm);
			\draw[black] (-0.6, -2.9) node {$2$};  
			\draw[black] (0.6, -2.9) node {$\Delta$};  
			\draw[black] (-1.5, -7.5) node {$\Delta$}; 
			\draw[black] (1.5, -7.5) node {$\delta$};
			\draw[black] (0, -8.5) node {$\varphi_3$}; 
			
			\begin{scope}[shift={(6,0)}]
			{\tikzstyle{every node}=[draw ,circle,fill=white, minimum size=0.5cm,
				inner sep=0pt]
				\draw[black,thick](0,-3) node (s1)  {$s_1$};
				\draw [black,thick](0, -5) node (u)  {$u$};
				\draw [black,thick](-1, -7) node (x)  {$x$};
				\draw [black,thick](1, -7) node (y)  {$y$};
			}
			\path[draw,thick,black]
			
			(s1) edge node[name=la,pos=0.4] {\color{black}\,\,\, $1$} (u)
			(u) edge node[name=la,pos=0.4] {\color{black}$\Delta$\quad\quad\,} (x)
			(u) edge node[name=la,pos=0.4] {\color{black}\,\,\,\,\,\,\, $\tau$} (y);
			
			
			\draw[dashed, black, line width=0.5mm] (s1)--++(140:1cm);
			\draw[dashed, black, line width=0.5mm] (s1)--++(40:1cm); 
			\draw[dashed, black, line width=0.5mm] (x)--++(200:1cm); 
			\draw[dashed, black, line width=0.5mm] (y)--++(340:1cm);
			\draw[black] (-0.6, -2.9) node {$2$};  
			\draw[black] (0.6, -2.9) node {$\Delta$};  
			\draw[black] (-1.5, -7.5) node {$\tau$}; 
			\draw[black] (1.5, -7.5) node {$\delta$};
			\draw[black] (0, -8.5) node {$\varphi_4$}; 
			\end{scope}

				\begin{scope}[shift={(12,0)}]
			{\tikzstyle{every node}=[draw ,circle,fill=white, minimum size=0.5cm,
				inner sep=0pt]
				\draw[black,thick](0,-3) node (s1)  {$s_1$};
				\draw [black,thick](0, -5) node (u)  {$u$};
				\draw [black,thick](-1, -7) node (x)  {$x$};
				\draw [black,thick](1, -7) node (y)  {$y$};
			}
			\path[draw,thick,black]
			
			(s1) edge node[name=la,pos=0.4] {\color{black}\,\,\, $\tau$} (u)
			(u) edge node[name=la,pos=0.4] {\color{black}$\Delta$\quad\quad\,} (x)
			(u) edge node[name=la,pos=0.4] {\color{black}\,\,\,\,\,\,\, $1$} (y);
			
			
			\draw[dashed, black, line width=0.5mm] (s1)--++(140:1cm);
			\draw[dashed, black, line width=0.5mm] (s1)--++(40:1cm); 
			\draw[dashed, black, line width=0.5mm] (x)--++(200:1cm); 
			\draw[dashed, black, line width=0.5mm] (y)--++(340:1cm);
			\draw[black] (-0.6, -2.9) node {$2$};  
			\draw[black] (0.6, -2.9) node {$\Delta$};  
			\draw[black] (-1.5, -7.5) node {$1$}; 
			\draw[black] (1.5, -7.5) node {$\delta$};
			
			\draw[black] (0, -8.5) node {$J_3$}; 
			\end{scope}

			\begin{scope}[shift={(18,0)}]

			{\tikzstyle{every node}=[draw ,circle,fill=white, minimum size=0.5cm,
				inner sep=0pt]
				\draw[black,thick](0,-3) node (s1)  {$s_1$};
				\draw [black,thick](0, -5) node (u)  {$u$};
				\draw [black,thick](-1, -7) node (x)  {$x$};
				\draw [black,thick](1, -7) node (y)  {$y$};
			}
			\path[draw,thick,black]
			
			(s1) edge node[name=la,pos=0.4] {\color{black}\,\,\, $\delta$} (u)
			(u) edge node[name=la,pos=0.4] {\color{black}$\Delta$\quad\quad\,} (x)
			(u) edge node[name=la,pos=0.4] {\color{black}\,\,\,\,\,\,\, $1$} (y);
			
			
			\draw[dashed, black, line width=0.5mm] (s1)--++(140:1cm);
			\draw[dashed, black, line width=0.5mm] (s1)--++(40:1cm); 
			\draw[dashed, black, line width=0.5mm] (x)--++(200:1cm); 
			\draw[dashed, black, line width=0.5mm] (y)--++(340:1cm);
			\draw[black] (-0.6, -2.9) node {$2$};  
			\draw[black] (0.6, -2.9) node {$\Delta$};  
			\draw[black] (-1.5, -7.5) node {$1$}; 
			\draw[black] (1.5, -7.5) node {$\tau$};
			\draw[black] (0, -8.5) node {$J_4$}; 
			\end{scope}
			\end{tikzpicture}
		\end{center}
		\vspace{-0.5cm}
		\caption{Coloring of $S(u; s_1, x, y)$}
		\label{J3}
	\end{figure}

\smallskip 
{\noindent \bf \setword{Subcase 2.1.1}{Case 1.1} : $\tau\prec \delta$ in $F_3$ with respect to $\varphi_5$.}
\smallskip 

	Let $s_{i}\in N_{\Delta-1}(r)$ such that $\pbar_5(s_{i})=\delta$.
	Since $y$ and $s_1$ are still $(2,\delta)$-linked with respect to 
	$\varphi_5$ and $\delta$ is 2-inducing of $F_3$, 
	by Lemma~\ref{thm:vizing-fan2} \eqref{thm:vizing-fan2-b}, 
	$r\in P_{s_{i}}(2,\delta,\varphi_5)$.  We reach a contradiction  through the following operations:
(1) a $(2,\delta)$-swap at $s_{i}$ (and so also at $r$); (2) a $(1,2)$-swap at $x$ and $s_{i}$;
and (3)  shift from $s_j$ to $s_{i}$, where we assume $\varphi_5(rs_j)=\tau$
for some $j\in [2,\Delta-2]$ and $s_j, s_{j+1}, \ldots, s_i$ is the 2-inducing sequence of $F_3$
starting at $s_j$ and ending at $s_i$.   
	Denote the new coloring by $\varphi_6$. 
	Now, $\pbar_6(r)=\tau$, $\varphi_6(s_1u)=\tau$, $\varphi_6(ux)=\Delta$, and $\pbar_6(x)=2$,
	and $K=(r,rs_1,s_1,s_1u, u, ux,x)$ is a Kierstead path with respect to $rs_1$
	and $\varphi_6$. Since $d_G(s_1)=\Delta-1$, 
	we get  a contradiction  to Lemma~\ref{Lemma:kierstead path1} (a) that $\{r,s_1,u,x\}$ is $\varphi_6$-elementary.
	
 \smallskip 
	{\noindent \bf Subcase 2.1.2: $\delta\prec \tau$ in $F_3$ with respect to $\varphi_5$.}
\smallskip 
	
	We only show that by performing Kempe changes, we can find an $(F_3,\varphi_5)$-stable coloring  such that the color on $S(u;s_1,x,y)$ with respect to it
	is as given in Figure~\ref{J3} $J_4$. Then the proof follows
	 the same ideas as in Case~\ref{Case 1.1} by exchanging the roles of $\tau$ and $\delta$. 
	Based on the coloring in Figure~\ref{J3} $J_3$,  do a $(1,\delta)$-swap at both $x$
	and $y$ and denote the resulting coloring by $\varphi_6$. 
	
	\begin{CLA}\label{claim2}
	$u\in P_y(1,\tau,\varphi_6)$.
	\end{CLA}

\proof[Proof of Claim~\ref{claim2}]
Assume to the contrary that $u\not\in P_y(1,\tau,\varphi_6)$.
Under this assumption, it must be the case that  
$u\in P_r(1,\tau,\varphi_6)$ (otherwise, performing a $(\delta,\Delta)$-swap at $x$ and a $(1,\tau)$-swap at $u$ shows that $s_1$ and $y$ are $(1,\Delta)$-linked,  showing a contradiction to Lemma~\ref{thm:vizing-fan1}~\eqref{thm:vizing-fan1b} that $s_1$
and $r$ are $(1,\Delta)$-linked).  
Since $u\in P_r(1,\tau,\varphi_6)$,  let $\varphi_7$ be obtained by 
doing a $(1,\tau)$-swap at $y$ and $(\delta,\Delta)$-swap at $x$, and let $F_3^*$
be the resulting multifan.  
Then  $P_{s_1}(\tau,\Delta,\varphi_7)=s_1uy$, 
implying that $\tau$ is a $\Delta$-inducing color of $F_3^*$ by  Lemma~\ref{thm:vizing-fan2} \eqref{thm:vizing-fan2-a}. 
Note that  $\delta$ is still a 2-inducing color of $F_3^*$
as the only operation that changes the color sequence of $F_3$ was the $(\delta,\Delta)$-swap we did to get $\varphi_7$ from $\varphi_6$. 
Thus, $\pbar_7^{-1}(\delta)$ and $\pbar_7^{-1}(\tau)$
are $(\delta,\tau)$-linked by Lemma~\ref{thm:vizing-fan2} \eqref{thm:vizing-fan2-a}. 
Also, since $\tau$ is $\Delta$-inducing and $\delta$ is 2-inducing of $F_3^*$, 
we know  $u\not\in P_y(\tau,\delta, \varphi_7)$. 
Since otherwise, after a $(\tau,\delta)$-swap at $y$,  $s_1$
and $y$ are $(\delta,\Delta)$-linked, 
showing a contradiction to Lemma~\ref{thm:vizing-fan2}~\eqref{thm:vizing-fan2-a}. 

Let $\varphi_8=\varphi_7/P_y(\tau, \delta,\varphi_7)$.   Now $P_y(\delta, \Delta, \varphi_8)=yux$.  Note also that $u\in P_{\pbar_8^{-1}(\delta)}(\delta,\tau, \varphi_8)=P_{\pbar^{-1}_8(\tau)}(\delta,\tau,\varphi_8)$. For otherwise, after a $(\tau,\delta)$-swap at $u$,  $s_1$
and $y$ are $(\delta,\Delta)$-linked, 
	showing a contradiction to the fact that $\delta$ is still a 2-inducing color of the resulting multifan. Note that $P_x(\delta,\Delta,\varphi_8)=xuy$. 
Let $\varphi_9=\varphi_8/P_x(\delta,\Delta,\varphi_8)$. 
Now  $\pbar^{-1}_9(\delta)$ and $\pbar^{-1}_9(\tau)$
are $(\delta,\tau)$-unlinked. However, since $\varphi_9$ is $(F_3^*,\varphi_8)$-stable,  $\tau$ is still a $\Delta$-inducing color and $\delta$
is 2-inducing of $F_3^{*}$ with respect to $\varphi_9$, we get a contradiction to Lemma~\ref{thm:vizing-fan2} \eqref{thm:vizing-fan2-a}.
\qed 
	
	Thus by Claim~\ref{claim2},  $u\in P_y(1,\tau, \varphi_6)$.   Do a $(1,\tau)$-swap at $y$ (and $u$), and  denote the resulting coloring by $\varphi_7$. 
	Note that $u\in P_x(1,\delta, \varphi_7)$ (as otherwise, after a $(1,\delta)$-swap at $x$, $s_1$ and $x$ are $(1,\Delta)$-linked,  showing a contradiction to Lemma~\ref{thm:vizing-fan1}~\eqref{thm:vizing-fan1b} that $s_1$
	and $r$ are $(1,\Delta)$-linked). Let $\varphi_8=\varphi_7/P_x(1,\delta, \varphi_7)$. Now with respect to $\varphi_8$, we have the coloring
	in Figure~\ref{J3} $J_4$. By the definition,  $\varphi_8$ is $(F_3,\varphi_5)$-stable 
	so we still have $\delta\prec \tau$ in $F_3$ with respect to $\varphi_8$. The remaining proof follows
	the same ideas as in Case~\ref{Case 1.1}.

	\smallskip 
	
	{\noindent \bf Subcase 2.2: $\pbar(x)=\Delta$ and $\pbar(y)=\Delta$.}
	\smallskip

	\begin{CLA}\label{only2-out-neighbors}
		We may assume that $|N_{\Delta-1}(u)\cap N_{\Delta-1}(r)|=\Delta-4$. 
	\end{CLA}
	
	\proof[Proof of Claim~\ref{only2-out-neighbors}] 
	We may assume that $x$ and $y$ are $(1,\Delta)$-linked. For otherwise, performing $(\Delta,1)-(1,2)$-swaps at $x$ reduces the problem to \ref{Case 1}. 
	Assume to the contrary that $|N_{\Delta-1}(u)\cap N_{\Delta-1}(r)|\le \Delta-5$.
	Then there exists $z\in N_{\Delta-1}(u)\setminus  N_{\Delta-1}(r)$ such that $z\ne x,y$. 
	Let $\pbar(z)=\lambda$. If $\lambda=2$, 
	by exchanging  the roles of $x$ and $z$,  we reduce the problem to \ref{Case 1}. 
	Thus, $\lambda \ne 2$.  Doing  $(\lambda,1)-(1,2)$-swaps at $z$ and exchanging  the roles of $x$ and $z$ reduces the problem to \ref{Case 1}. 
	\qed 
	
	\begin{CLA}\label{2-sequence-fan}
		We  may assume that $F(r,s_1:s_\alpha:s_{\Delta-2})$ is a typical multifan with  two sequences.  
		That is, $F$ contains both $2$-inducing sequence  and $\Delta$-inducing sequence.
	\end{CLA}
	\proof[Proof of Claim~\ref{2-sequence-fan}]
	Recall that $F_\varphi(r,s_1: s_\alpha:s_{\Delta-2})$ is a typical multifan. 
	As $\Delta\ge 7$, $|N_{\Delta-1}(u)\cap N_{\Delta-1}(r)|=\Delta-4\ge 3$ 
	by Claim~\ref{only2-out-neighbors}. If $F$
	is a typical 2-inducing  multifan, then let $s_i\in N_{\Delta-1}(u)\cap N_{\Delta-1}(r)$
	such that $s_i\ne s_1$ and that  $\pbar(s_i)$ is not the last 2-inducing color of $F$. 
	Then we shift from $s_2$ to $s_{i-1}$, uncolor $rs_i$, and color $rs_1$ by 2. 
	Now $F^*=(r, rs_i, s_i, rs_{i+1}, s_{i+1}, \ldots, rs_{\Delta-2}, s_{\Delta-2}, rs_{i-1}, s_{i-1}, \ldots, rs_1, s_1)$
	is a multifan with  two sequences.  By permuting the name of colors 
	and the label of vertices in $\{s_1,\ldots, s_{\Delta-2}\}$, we can assume that 
	$F=F^*$ is a typical multifan with  two sequences.  
	\qed 
	
	Let $\varphi(s_1u)=\delta$, $\varphi(ux)=\tau$,  and $\varphi(uy)=\lambda$.  By exchanging the roles of 
	the two colors 2 and $\Delta$,  we have two possibilities for $\varphi(uy)$: (A)
	 $\varphi(uy)=\lambda=1$; and (B) 
 $\varphi(uy)=\lambda$ is 2-inducing. 
	(When $\varphi(uy)$ is $\Delta$-inducing, we will first assume that $\pbar(x)=2$ and $\pbar(y)=2$ (by performing $(\Delta,1)-(1,2)$-swaps at both $x$ and $y$). Then all the argument will be symmetric to the argument for the  case (B) above.) 
	We now consider two cases to finish the proof. 
	
	\smallskip 
	{\noindent \bf \setword{Subcase 2.2.1}{Case 2.1}: $\lambda$ is not the last 2-inducing color of $F$.}
	\smallskip 
	
	We first perform $(\Delta,\lambda)-(\lambda,1)$-swaps at both $x$ and $y$. 
	Denote by $\varphi_1$ the resulting coloring and $F_1$ the corresponding multifan. 
	Since $\lambda$ is not the last 2-inducing color of $F$,  $F_1$ still has  two sequences with respect to $\varphi_1$. 
	The current coloring of $S(u;s_1,x,y)$ is given in Figure~\ref{L1} $L_1$. 
	Since $s_1$ and $r$
	are $(1,\Delta)$-linked by Lemma~\ref{thm:vizing-fan1}~\eqref{thm:vizing-fan1b}, $\delta\ne 1$. 
	We next show $u\in P_y(1,\delta, \varphi_1)$ that will lead to the coloring in Figure~\ref{L1} $L_2$ after a $(1,\delta)$-swap at both $x$ and $y$. 
	
	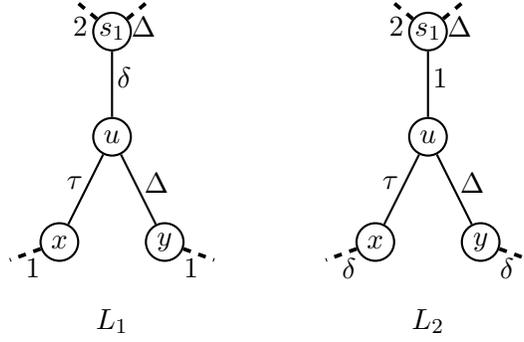
\begin{figure}[!htb]
		\begin{center}
			\begin{tikzpicture}[scale=0.7]
			
			{\tikzstyle{every node}=[draw ,circle,fill=white, minimum size=0.5cm,
				inner sep=0pt]
				\draw[black,thick](0,-3) node (s1)  {$s_1$};
				\draw [black,thick](0, -5) node (u)  {$u$};
				\draw [black,thick](-1, -7) node (x)  {$x$};
				\draw [black,thick](1, -7) node (y)  {$y$};
			}
			\path[draw,thick,black]
			
			(s1) edge node[name=la,pos=0.4] {\color{black}\,\,\, $\delta$} (u)
			(u) edge node[name=la,pos=0.4] {\color{black}$\tau$\quad\quad} (x)
			(u) edge node[name=la,pos=0.4] {\color{black}\,\,\,\,\,\,\, $\Delta$} (y);
			
			
			\draw[dashed, black, line width=0.5mm] (s1)--++(140:1cm);
			\draw[dashed, black, line width=0.5mm] (s1)--++(40:1cm); 
			\draw[dashed, black, line width=0.5mm] (x)--++(200:1cm); 
			\draw[dashed, black, line width=0.5mm] (y)--++(340:1cm);
			\draw[black] (-0.6, -2.9) node {$2$};  
			\draw[black] (0.6, -2.9) node {$\Delta$};  
			\draw[black] (-1.5, -7.5) node {$1$}; 
			\draw[black] (1.5, -7.5) node {$1$};
			
			\draw[black] (0, -8.5) node {$L_1$}; 
			
			\begin{scope}[shift={(6,0)}]

			{\tikzstyle{every node}=[draw ,circle,fill=white, minimum size=0.5cm,
				inner sep=0pt]
				\draw[black,thick](0,-3) node (s1)  {$s_1$};
				\draw [black,thick](0, -5) node (u)  {$u$};
				\draw [black,thick](-1, -7) node (x)  {$x$};
				\draw [black,thick](1, -7) node (y)  {$y$};
			}
			\path[draw,thick,black]
			
			(s1) edge node[name=la,pos=0.4] {\color{black}\,\,\, $1$} (u)
			(u) edge node[name=la,pos=0.4] {\color{black}$\tau$\quad\quad} (x)
			(u) edge node[name=la,pos=0.4] {\color{black}\,\,\,\,\,\,\, $\Delta$} (y);
			
			
			\draw[dashed, black, line width=0.5mm] (s1)--++(140:1cm);
			\draw[dashed, black, line width=0.5mm] (s1)--++(40:1cm); 
			\draw[dashed, black, line width=0.5mm] (x)--++(200:1cm); 
			\draw[dashed, black, line width=0.5mm] (y)--++(340:1cm);
			\draw[black] (-0.6, -2.9) node {$2$};  
			\draw[black] (0.6, -2.9) node {$\Delta$};  
			\draw[black] (-1.5, -7.5) node {$\delta$}; 
			\draw[black] (1.5, -7.5) node {$\delta$};
			\draw[black] (0, -8.5) node {$L_2$}; 
			\end{scope}
			\end{tikzpicture}
		\end{center}
		\vspace{-0.5cm}
		\caption{Coloring of $S(u; s_1, x, y)$}
		\label{L1}
	\end{figure}
	
	\begin{CLA}\label{u-in-p-y}
		$u\in P_y(1,\delta,\varphi_1)$.
	\end{CLA}
	\proof[Proof of Claim~\ref{u-in-p-y}]
	Assume to the contrary that $u\not\in P_y(1,\delta,\varphi_1)$. 
	This implies that  $\delta$ is a $\Delta$-inducing color of $F_1$ (since after doing a $(1,\delta)$-swap at $y$,
	  $s_1$ and $y$ are $(\delta,\Delta)$-linked). 
	If $\tau$ is a $\Delta$-inducing of $F_1$, then we let $\varphi_2$ be obtained by performing $(1,2)-(2,\tau)-(1,\tau)$-swaps 
	at both $x$ and $y$ based on the coloring of $L_1$ in Figure~\ref{L1}. 
		Now, we must have that $u\in P_x(1,\delta,\varphi_2)$ or $u\in P_y(1,\delta,\varphi_2)$ since $\delta$
	is either $2$-inducing or $\Delta$-inducing with respect to $F_1$.  Let $\varphi_3$ be obtained from  $\varphi_2$ by 
	performing a $(1,\delta)$-swap at both $x$ and $y$. 
	Then both $K_1=(r,rs_1,s_1,s_u,u,ux,x)$ and $K_2=(r,rs_1,s_1,s_u,u,uy,y)$
	are Kierstead paths with respect to $rs_1$ and $\varphi_3$. 
	Since $d_G(s_1)=\Delta-1$, applying Lemma~\ref{Lemma:kierstead path1} (b), 
	$x$ and $s_1$ are $(\delta, \Delta)$-linked and $y$ and $s_1$ are $(2,\delta)$-linked. 
	However, by Lemma~\ref{thm:vizing-fan2}~\eqref{thm:vizing-fan2-a}, 
	$s_1$ and $\pbar^{-1}_3(\delta)$  are either $(\delta,2)$ or $(\delta, \Delta)$-linked, showing a contradiction. 
	
	Thus we assume  that  $\tau$ is a $2$-inducing color of $F_1$.  
	Based on the coloring of $S(u;s_1,x,y)$ as given in Figure~\ref{L1} $L_1$, 
	we perform $(1,\tau)-(\tau,\delta)$-swaps at both $x$ and $y$ and 
	let $\varphi_2$ be the resulting coloring. Note that either $\varphi_2(s_1u)=\delta$ or 
	$\varphi_2(s_1u)=\tau$. 
	If $\varphi_2(s_1u)=\delta$, then  after doing  a $(1,\delta)$-swap at both $x$
	and $y$, $s_1$ and $y$ are $(1,\Delta)$-linked, which gives  a contradiction to 
	Lemma~\ref{thm:vizing-fan1}~\eqref{thm:vizing-fan1b} that $s_1$ and $r$
	are $(1,\Delta)$-linked. Thus $\varphi_2(s_1u)=\tau$. 
	We first do a $(1,\delta)$-swap at both $x$ and $y$. 
	Then since $\tau$ is a 2-inducing color of $F_1$, $u\in P_y(1,\tau,\varphi_2)$ (since otherwise, after doing a $(1,\tau)$-swap at $y$,  $s_1$ and $y$ are $(\tau,\Delta)$-linked, showing a contradiction to  Lemma~\ref{thm:vizing-fan2} \eqref{thm:vizing-fan2-a}).
	Thus we do a $(1,\tau)$-swap at both $x$ and $y$ and let $\varphi_3$ be the new coloring.  
	Note that $\delta$ is still $\Delta$-inducing and $\tau$ 
	is 2-inducing with respect to $F_1$ and $\varphi_3$.  Thus  $\pbar^{-1}_3(\delta)$ and $\pbar^{-1}_3(\tau)$
	are $(\delta,\tau)$-linked by Lemma~\ref{thm:vizing-fan2} \eqref{thm:vizing-fan2-a}. 
	Let $\varphi_4$ be obtained from $\varphi_3$ by 
	doing  a $(\delta,\tau)$-swap at $y$, and let $F_1^*$ be the resulting multifan.  
	Then $K=(r,rs_1,s_1,s_1u, u, uy, y)$ is a Kierstead path with respect to 
	$rs_1$ and $\varphi_4$.  Since $d_G(s_1)=\Delta-1$, applying Lemma~\ref{Lemma:kierstead path1} (b), 
	$y$ and $s_1$ are $(2,\delta)$-linked. 
	Since $\delta$ is still $\Delta$-inducing and $\tau$ 
	is 2-inducing with respect to $F_1^*$ and $\varphi_4$, we achieve a contradiction to the fact that $s_1$
	and $\pbar^{-1}_4(\delta)$ are $(2,\delta)$-linked by Lemma~\ref{thm:vizing-fan2}~\eqref{thm:vizing-fan2-a}. 
	Therefore it must be the case  $u\in P_y(1,\delta,\varphi_1)$.
	\qed 
	
	Since $u\in P_y(1,\delta, \varphi_1)$, we perform a $(1,\delta)$-swap at both $x$ and $y$
	gives $L_2$ in Figure~\ref{L1}. Call the resulting coloring $\varphi_2$. 
	Now $K=(r,rs_1,s_1,s_u, u,uy,y)$ is a Kierstead path with respect to $rs_1$ 
	and $\varphi_2$. 
	Since $d_G(s_1)=\Delta-1$, by Lemma~\ref{Lemma:kierstead path1} (b), $y$
	and $s_1$ are $(2,\delta)$-linked. 
	It deduces that $\delta$ must be a 2-inducing color of $F_1$ with respect to $\varphi_2$. 
	Recall that $F_1$ still has two sequences with respect to $\varphi_2$.  Let $\gamma$ be a $\Delta$-inducing color 
	of $F_1$. Since $\pbar^{-1}_2(\delta)$ and $\pbar^{-1}_2(\gamma)$
	are $(\delta,\gamma)$-linked by Lemma~\ref{thm:vizing-fan2}~\eqref{thm:vizing-fan2-a}, 
	we do a $(\delta,\gamma)$-swap at $y$ to get $\varphi_3$. Still, 
		$\delta$ is a 2-inducing color and $\gamma$ is a $\Delta$-inducing color 
		of the resulting multifan. By Lemma~\ref{Lemma:kierstead path1} (b), 
 $s_1$ and $y$ are $(2,\gamma)$-linked, showing a 
	contradiction to the fact that $s_1$ and $\pbar^{-1}_3(\gamma)$ are $(2,\gamma)$-linked. 
	
	\smallskip 
	{\noindent \bf \setword{Subcase 2.2.2}{Case 2.2}: $\lambda$ is the last 2-inducing color of $F$.}
	\smallskip 
	
	If $\tau$ is $2$-inducing, then $\tau \prec \lambda$. 
	This gives back to the previous case by exchanging the roles of $\tau$ and $\lambda$.  
	If $\tau$ is $\Delta$-inducing and $\tau$ 
	is not the last $\Delta$-inducing color, then 
	by doing $(\Delta,1)-(1,2)$-swaps at $x$ and $y$, a similar 
	proof follows as in the previous case by exchanging the roles of 2 and $\Delta$. 
	Thus $\tau$ is the last $\Delta$-inducing color of $F$.

	Let $C_u$ be the cycle in $G_\Delta$  containing $u$.  By Theorem~\ref{Thm:vizing-fan2b} \eqref{common2},  for every vertex on $C_u$, its $(\Delta-1)$-neighborhood is $N_{\Delta-1}(u)$. 
	As $|V(C_u)|\ge 3$, there exist $u^*, u'\in V(C_u)\setminus\{u\}$
	such that one of $\varphi(u^*y)$ and $ \varphi(u'y)$ is neither $\tau$ nor $\lambda$. 
	Assume that $\varphi(u^*y)\not\in \{\tau, \lambda\}$. 
	Letting $u^*$ play the role of $u$, we reduce the problem to the previous case, 	
finishing proof of Theorem~\ref{Thm:nonadj_Delta_vertex}.
\end{proof}


\end{document}